\documentclass[11pt, article]{amsart}
\usepackage{amsmath, amsthm, amscd, amsfonts, amssymb, graphicx, color}
\usepackage[bookmarksnumbered, colorlinks, plainpages]{hyperref}
\usepackage[all]{xy}

\hypersetup{colorlinks=true,linkcolor=red, anchorcolor=green, citecolor=cyan, urlcolor=red, filecolor=magenta, pdftoolbar=true}
\newtheorem{theorem}{Theorem}[section]
\newtheorem{lemma}[theorem]{Lemma}
\newtheorem{proposition}[theorem]{Proposition}
\newtheorem{corollary}[theorem]{Corollary}

\theoremstyle{definition}
\newtheorem{definition}[theorem]{Definition}

\theoremstyle{remark}
\newtheorem{remark}[theorem]{Remark}

\numberwithin{equation}{section}
\newtheorem*{theorem*}{Theorem}
\allowdisplaybreaks
\begin{document}
\title[A new seminorm of $n$-tuple operators and its applications] 
{ A new seminorm of $n$-tuple operators and its applications  } 


\author[ P. Bhunia  and M. Guesba]{ Pintu Bhunia and Messaoud Guesba}

\address[P.~Bhunia]{Department of Mathematics, Indian Institute of
Science, Bangalore 560012, India}
\email{\tt pintubhunia5206@gmail.com; pintubhunia@iisc.ac.in}

\address[M. Guesba]{Faculty of Exact Sciences, Department of Mathematics\\
El Oued University, 39000 Algeria}
\email{guesbamessaoud2@gmail.com, guesba-messaoud@univ-eloued.dz}

\thanks{Pintu Bhunia was supported by National Post-Doctoral Fellowship PDF/2022/000325 from SERB (Govt.\ of India) and SwarnaJayanti Fellowship SB/SJF/2019-20/14 (PI: Apoorva Khare) from SERB (Govt.\ of India).
He also would like to thank National Board for Higher Mathematics (Govt.\ of India) for the financial support in the form of NBHM Post-Doctoral Fellowship 0204/16(3)/2024/R\&D-II/6747.
}

\maketitle

\begin{abstract}
We introduce a new seminorm of $n$-tuple operators,
which generalizes the $A$-Euclidean operator radius of $n$-tuple bounded linear operators on a complex Hilbert space. We
introduce and study basic properties of this seminorm. As an application of
the present study, we estimate bounds for the $A$-Euclidean operator radius ($A$-joint numerical radius). In addition, we improve on
some of the important existing $A$-numerical radius inequalities
and related results.
\end{abstract}

\bigskip

\bigskip
\section{\protect\Large Introduction and Preliminaries }

Throughout this paper, ${\mathcal{H}}$ denotes a non-trivial complex Hilbert
space with inner product $\left\langle .,.\right\rangle $ and the associated
norm $\left\Vert \cdot\right\Vert $. Let ${\mathcal{B}}({\mathcal{H}})$ be the $%
C^{\ast }$-algebra of all bounded linear operators on ${\mathcal{H}}$. An
operator $A\in \mathcal{B}(\mathcal{H})$ is called positive if $\left\langle
Ax,x\right\rangle \geq 0$ for all $x\in \mathcal{H}$, and we denote it $%
A\geq 0$. For every operator $T\in $ ${\mathcal{B}}({\mathcal{H}})$, ${%
\mathcal{N}}\left( T\right) $, ${\mathcal{R}}\left( T\right) $ and $%
\overline{{\mathcal{R}}\left( T\right) }$ stand for the null
space, the range and the closure of the range of $T$, respectively, and $T^{\ast }$ stands for the adjoint of $T.$
Any positive operator $A$ defines a positive semi-definite sesquilinear form
\begin{equation*}
\left\langle .,.\right\rangle _{A}:{\mathcal{H}}\times {\mathcal{H}}%
\rightarrow 
\mathbb{C}
\text{, \ }\left\langle x,y\right\rangle _{A}:=\left\langle
Ax,y\right\rangle \text{.}
\end{equation*}
This semi-inner product induces a seminorm $\left\Vert
.\right\Vert _{A}$, which is defined by
$\left\Vert x\right\Vert _{A}:=\left\langle Ax,x\right\rangle ^{\frac{1}{2}%
}=\left\Vert A^{{1}/{2}}x\right\Vert \text{.}
$ 
Observe that $\left\Vert x\right\Vert _{A}=0$ if and only if $x\in {\mathcal{%
N}}\left( A\right) $. Therefore, $\left\Vert .\right\Vert _{A}$ is a norm on ${%
\mathcal{H}}$ if and only if $A$ is an injective operator. Also, the
semi-normed space $\left( {\mathcal{B}}({\mathcal{H}}),\left\Vert
.\right\Vert _{A}\right) $ is complete if and only if ${\mathcal{R}}\left(
A\right) $ is closed.
An operator $W\in {\mathcal{B}}({%
\mathcal{H}})$\ is called an $A$-adjoint of $T$ if
$\left\langle Tx,y\right\rangle _{A}=\left\langle x,Wy\right\rangle _{A}\text{
for very }x,y\in {\mathcal{H}}\text{,}
$
or equivalently, 
$AW=T^{\ast }A\text{.}
$
An operator $T$ is called $A$-self-adjoint if $AT=T^{\ast }A$, and it is
called $A$-positive if $AT$ is positive, we write $T\geq _{A}0$.
The existence of an $A$-adjoint operator is not guaranteed. The set of all
operators which admit $A$-adjoints is denoted by ${\mathcal{B}}_{A}({%
\mathcal{H}})$. By Douglas theorem \cite{Dou}, we get%
\begin{eqnarray*}
{\mathcal{B}}_{A}({\mathcal{H}}) &=&\left\{ T\in {\mathcal{B}}({\mathcal{H}}%
):{\mathcal{R}}\left( T^{\ast }A\right) \subseteq {\mathcal{R}}\left( A\right)
\right\} \\
&=&\left\{ T\in {\mathcal{B}}({\mathcal{H}}):\exists \text{ }c>0 \textit{ such that }\left\Vert
ATx\right\Vert \leq c\left\Vert Ax\right\Vert ,\forall \text{ }x\in {%
\mathcal{H}}\right\} \text{.}
\end{eqnarray*}
If $T\in {\mathcal{B}}_{A}({\mathcal{H}})$, then $T$ admits an $A$-adjoint. Moreover, there exists a distinguished $A$-adjoint of $T$, namely, the reduced solution of the equation $AX=T^{\ast }A$, i.e., $%
T^{\sharp _{A}}=A^{\dag }T^{\ast }A$, where $A^{\dag }$ is the Moore-Penrose
inverse of $A$. The $A$-adjoint operator $T^{\sharp _{A}}$ verifies 
\begin{equation*}
AT^{\sharp _{A}}=T^{\ast }A\text{, }{\mathcal{R}}\left( T^{\sharp
_{A}}\right) \subseteq \overline{{\mathcal{R}}\left( A\right) }\text{ and }{%
\mathcal{N}}\left( T^{\sharp _{A}}\right) ={\mathcal{N}}\left( T^{\ast
}A\right) \text{.}
\end{equation*}
Again, by applying Douglas theorem \cite{Dou}, we have
\begin{equation*}
{\mathcal{B}}_{A^{{1}/{2}}}({\mathcal{H}})=\left\{ T\in {\mathcal{B}}({%
\mathcal{H}}):\exists \text{ }c>0 \textit{ such that } \left\Vert Tx\right\Vert _{A}\leq
c\left\Vert x\right\Vert _{A},\forall \text{ }x\in {\mathcal{H}}\right\} 
\text{.}
\end{equation*}
Any operator in ${\mathcal{B}}_{A^{{1}/{2}}}({\mathcal{H}})$ is called
 $A$-bounded operator. Moreover, it was given in \cite{AR} that if $T\in {%
\mathcal{B}}_{A^{{1}/{2}}}({\mathcal{H}})$, then
$\left\Vert T\right\Vert _{A}:=\sup\limits_{x\notin {\mathcal{N}}\left(
A\right) }\frac{\left\Vert Tx\right\Vert _{A}}{\left\Vert x\right\Vert _{A}}%
=\sup\limits_{\substack{ x\in {\mathcal{H}}  \\ \left\Vert x\right\Vert
_{A}=1 }}\left\Vert Tx\right\Vert _{A}\text{.}$
In addition, if $T$ is $A$-bounded, then $T\left( {\mathcal{N}}\left(
A\right) \right) \subseteq $ ${\mathcal{N}}\left( A\right) $ and
$\left\Vert Tx\right\Vert _{A}\leq \left\Vert T\right\Vert _{A}\left\Vert
x\right\Vert _{A}\text{, }\forall \text{ }x\in {\mathcal{H}}\text{.}
$   
For $T\in {\mathcal{B}}_{A}({\mathcal{H}})$, the following statements
hold:
\begin{itemize}

   \item  If $AT=TA$ then $T^{\sharp _{A}}=P_{\overline{\Re \left( A\right) }
}T^{\ast }$, where $P_{\overline{\Re \left( A\right) }
}$ is the orthogonal projection onto ${\overline{\Re \left( A\right) }
}$.

   \item  $T^{\sharp _{A}}\in {\mathcal{B}}_{A}({\mathcal{H}})$, $\left( T^{\sharp
_{A}}\right) ^{{\sharp _{A}}}=P_{\overline{\Re \left( A\right) }}TP_{%
\overline{\Re \left( A\right) }}$ and $\left( \left( T^{\sharp _{A}}\right)
^{\sharp _{A}}\right) ^{^{\sharp _{A}}}=T^{\sharp _{A}}$.

   \item  $T^{\sharp_A}T$ and $TT^{\sharp_A}$ are $A$-self-adjoint and $A$-positive operators.

   \item  If $S\in {\mathcal{B}}_{A}({\mathcal{H}})$, then $TS\in {\mathcal{B}}_{A}(%
{\mathcal{H}})$ and $\left( TS\right) ^{\sharp _{A}}=S^{\sharp
_{A}}T^{^{\sharp _{A}}}$.

   \item  $\left\Vert T\right\Vert _{A}=$ $\left\Vert T^{\sharp _{A}}\right\Vert
_{A}$ $=\left\Vert T^{\sharp _{A}}T\right\Vert _{A}^{{1}/{2}}=\left\Vert
TT^{\sharp _{A}}\right\Vert _{A}^{{1}/{2}}$.
\end{itemize}
Furthermore, for all $T,S\in {\mathcal{B}}_{A^{{1}/{2}}}({\mathcal{H}})$, we have
$\left\Vert TS\right\Vert _{A}\leq \left\Vert T\right\Vert _{A}\left\Vert
S\right\Vert _{A}\text{.} $ 


The numerical radius of $T\in {\mathcal{B}}({%
\mathcal{H}})$ is defined as
$\omega \left( T\right) :=\sup \left\{ | \left\langle Tx,x\right\rangle |:x\in {%
\mathcal{H}}\text{, }\left\Vert x\right\Vert =1\right\} \text{.}
$
There have been many generalizations of the numerical radius. One of these
generalizations is the $A$-numerical radius of an operator $T\in {\mathcal{B}%
}_{A^{1/2}}({\mathcal{H}})$, which is defined as
\begin{equation*}
\omega _{A}\left( T\right) :=\sup \left\{ \left\vert \left\langle
Tx,x\right\rangle _{A}\right\vert :x\in {\mathcal{H}}\text{, }\left\Vert
x\right\Vert _{A}=1\right\} \text{.}
\end{equation*}
Obviously, if $A=I$, then $\omega _{A}\left( T\right) =\omega \left( T\right) $.
One basic fact about the $A$-numerical radius is
$\omega _{A}\left( T\right) =\omega _{A}\left( T^{\sharp _{A}}\right) \text{
for every }T\in {\mathcal{B}}_{A}({\mathcal{H}})\text{.}
$  
Further, $A$-numerical radius is a seminorm on ${\mathcal{B}}_{A^{1/2}}({\mathcal{
H}})$ and for every $T\in {\mathcal{B}}_{A^{{1}/{2}}}({\mathcal{H}}%
) $, we have
\begin{equation}
\frac{1}{2}\left\Vert T\right\Vert _{A}\leq \omega _{A}\left( T\right) \leq
\left\Vert T\right\Vert _{A}\text{.}  \label{00}
\end{equation}
If $T\in {\mathcal{B}}_{A}({\mathcal{H}})$ is $A$-self-adjoint, then 
$\omega _{A}\left( T\right) =\left\Vert T\right\Vert _{A}\text{}$. And if $AT^2=0$, then $\frac{1}{2}\left\Vert T\right\Vert _{A}= \omega _{A}\left( T\right).$
For proofs and more facts about the $A$-numerical radius of operators, we refer to \cite{Bhunia-IJPA, Bhunia-book, BH, BH2, GU2, GU3, GU4, Z}.
A refinement of the second inequality in (\ref{00}) was established by
Zamani \cite{Z}, namely, if $T\in {\mathcal{B}}_{A}({
\mathcal{H}})$, then 
\begin{equation}
\omega _{A}^{2}(T)\leq \frac{1}{2}\left\Vert T^{\sharp _{A}}T+TT^{\sharp
_{A}}\right\Vert _{A}\text{.}  \label{2}
\end{equation}
In \cite{Z}, he also proved that 
\begin{equation}
\omega _{A}^{2}(T)\leq \frac{1}{4}\left\Vert T^{\sharp _{A}}T+TT^{\sharp
_{A}}\right\Vert _{A}+\frac{1}{2}\omega _{A}\left( T^{2}\right) \text{.}  \label{4}
\end{equation}
After that Guesba \cite{GU2} generalized (\ref{2}) as
\begin{equation}
\omega _{A}^{2n}(T)\leq \frac{1}{2} \left\Vert \left( T^{\sharp _{A}}T\right)
^{n}+\left( TT^{\sharp _{A}}\right) ^{n} \right\Vert _{A} \quad \text{ for all }n\in 
\mathbb{N}
\text{.}  \label{3}
\end{equation}
He also established an upper bound for the
product of two operators $T,S\in {\mathcal{B}}_{A}({\mathcal{H}
})$ as
\begin{equation}
\omega _{A}^{n}(S^{\sharp _{A}}T)\leq \frac{1}{2} \left\Vert \left( T^{\sharp
_{A}}T\right) ^{n}+\left( S^{\sharp _{A}}S\right) ^{n} \right\Vert _{A} \quad \text{ for
all }n\in 
\mathbb{N}
\text{.}  \label{6}
\end{equation}
In particular, for $n=1$, it follows that
\begin{equation}
\omega _{A}(S^{\sharp _{A}}T)\leq \frac{1}{2} \left\Vert T^{\sharp
_{A}}T+S^{\sharp _{A}}S \right\Vert _{A}\text{.}  \label{7}
\end{equation}

\noindent  For every $n\in 
\mathbb{N}
$, let ${\mathcal{B}}_{A^{{1}/{2}}}({\mathcal{H}})^{n}$ denote the
product of $n$-copies of ${\mathcal{B}}_{A^{{1}/{2}}}({\mathcal{H}})$,
that is 
\begin{equation*}
{\mathcal{B}}_{A^{{1}/{2}}}({\mathcal{H}})^{n}:=\left\{ \mathbf{T=}%
\left( T_{1},\ldots ,T_{n}\right) :T_{1},\ldots,T_{n}\in {\mathcal{B}}_{A^{1/2}}({\mathcal{H}})\right\}.
\end{equation*}
Similarly, let ${\mathcal{B}}_{A}({\mathcal{H}})^{n}$ denote the
product of $n$-copies of ${\mathcal{B}}_{A^{}}({\mathcal{H}})$,
that is 
\begin{equation*}
{\mathcal{B}}_{A}({\mathcal{H}})^{n}:=\left\{ \mathbf{T=}\left(
T_{1},\ldots,T_{n}\right) :T_{1},\ldots,T_{n}\in {\mathcal{B}}_{A}({\mathcal{H}}%
)\right\} \text{.}
\end{equation*}


\noindent For $\mathbf{T}=\left(
T_{1},\cdots ,T_{n}\right)$, $\mathbf{S=}\left( S_{1},\cdots ,S_{n}\right)
\in {\mathcal{B}}_{A^{{1}/{2}}}\left( {\mathcal{H}}\right) ^{n}$, we
denote $\mathbf{T}+\mathbf{S}:=\left( T_{1}+S_{1},\cdots ,T_{n}+S_{n}\right) 
$, $\mathbf{TS}:=\left( T_{1}S_{1},\cdots ,T_{n}S_{n}\right) $ and $\lambda 
\mathbf{T:}=\left( \lambda T_{1},\cdots ,\lambda T_{n}\right) $ for any
scalar $\lambda \in 
\mathbb{C}
$. For $\mathbf{T}=\left( T_{1},\cdots ,T_{n}\right) $ in ${ \mathcal{B}}_{A}({\mathcal{H}})^{n}$, we denote $\mathbf{T}^{\sharp _{A}}=\left( T_{1}^{\sharp _{A}},\cdots ,T_{n}^{\sharp _{A}}\right) $.

\noindent The $A$-joint operator norm of an $n$-tuple $\mathbf{T=}\left(
T_{1},\ldots,T_{n}\right) \in {\mathcal{B}}_{A^{{1}/{2}}}\left( {\mathcal{H}%
}\right) ^{n}$ was defined in \cite{GU1} by 
\begin{equation*}
\left\Vert \mathbf{T}\right\Vert _{A}:=\sup \left\{ \left(
\sum\limits_{k=1}^{n}\left\Vert T_{k}x\right\Vert _{A}^{2}\right) ^{\frac{1}{%
2}}:x\in {\mathcal{H}}\text{, }\left\Vert x\right\Vert _{A}=1\text{ }%
\right\} \text{.}
\end{equation*}
The $A$-Euclidean operator radius of $\mathbf{T=}\left(
T_{1},\cdots ,T_{n}\right) \in {\mathcal{B}}_{A^{{1}/{2}}}\left( {%
\mathcal{H}}\right) ^{n}$ was defined in \cite{BA} by 
\begin{equation*}
\omega _{A}\left( \mathbf{T}\right) :=\sup \left\{ \left(
\sum\limits_{k=1}^{n}\left\vert \left\langle T_{k}x,x\right\rangle
_{A}\right\vert ^{2}\right) ^{{1}/{2}}:x\in {\mathcal{H}}\text{, }%
\left\Vert x\right\Vert _{A}=1\text{ }\right\} \text{.}
\end{equation*}
For the case $n=1$, we get the classical $A$-numerical radius.
We define the $A$-joint Crawford number of $\mathbf{T=}\left( T_{1},\cdots
,T_{n}\right) \in {\mathcal{B}}_{A^{{1}/{2}}}\left( {
\mathcal{H}}\right) ^{n}$ as
\begin{equation*}
c_{A}\left( \mathbf{T}\right) :=\inf \left\{ \left(
\sum\limits_{k=1}^{n}\left\vert \left\langle T_{k}x,x\right\rangle
_{A}\right\vert ^{2}\right) ^{{1}/{2}}:x\in {\mathcal{H}}\text{, }%
\left\Vert x\right\Vert _{A}=1\text{ }\right\} \text{.}
\end{equation*}
In \cite{BH1}, Bhunia et al. introduced a new seminorm (called $A_{\alpha }$-seminorm) on ${\mathcal{B}}_{A^{{1}/{2}}}\left( {\mathcal{H}}\right) $ as:
\begin{equation*}
\left\Vert T\right\Vert _{A_{\alpha }}:=\sup \left\{  \left( \alpha
\left\vert \left\langle Tx,x\right\rangle _{A}\right\vert ^{2}+\left(
1-\alpha \right) \left\Vert Tx\right\Vert _{A}^{2} \right) ^{\frac{1}{%
2}}:x\in {\mathcal{H}}\text{, }\left\Vert x\right\Vert _{A}=1\text{ }%
\right\} \text{,} \quad 
\end{equation*}
where $0\leq \alpha \leq 1$, and studied various properties of this $A_{\alpha }$-seminorm and its applications.

\smallskip

In this paper, we introduce (motivated by the $A_{\alpha }$-seminorm and other seminorms studied in \cite{Bhunia-GMJ, BH1, GU1, GU5, J, MOS, Sain}) a
new seminorm on ${\mathcal{B}}_{A^{{1}/{2}}}\left( {\mathcal{H}}
\right) ^{n}$, which generalizes $A_{\alpha }$-seminorm. 
 In Section \ref{sec2}, we collect a few
lemmas that are required to state and prove the results in the subsequent
sections. In Section \ref{sec3}, we introduce and study basic properties of this
seminorm. As an application of this seminorm, we estimate some upper
and lower bounds for the $A$-Euclidean operator radius. In addition, we improve
some of the important related inequalities \eqref{2}--\eqref{7}.

\section{Auxiliary lemmas}\label{sec2}

Here, we present the following lemmas that will be used to
develop our main results.

\begin{lemma}
\label{L1}\bigskip \cite{SAD} Let $x,y,z\in {\mathcal{H}}$ with $\left\Vert
z\right\Vert _{A}=1$. Then%
\begin{equation*}
\left\vert \left\langle x,z\right\rangle _{A}\left\langle z,y\right\rangle
_{A}\right\vert \leq \frac{1}{2}\left( \left\Vert x\right\Vert
_{A}\left\Vert y\right\Vert _{A}+\left\vert \left\langle x,y\right\rangle
_{A}\right\vert \right) \text{.}
\end{equation*}
\end{lemma}

\begin{lemma}
\label{L2}\bigskip \cite{CO} Let $T\in {\mathcal{B}}\left( {\mathcal{H}}%
\right) $ be $A$-positive and let $x\in {\mathcal{H}}$ with $%
\left\Vert x\right\Vert _{A}=1$. Then
\end{lemma}

\begin{equation*}
\left\langle Tx,x\right\rangle _{A}^{n}\leq \left\langle
T^{n}x,x\right\rangle _{A}\text{ for all }n\in 
\mathbb{N}
\text{.}
\end{equation*}

\begin{lemma}
\label{L4}\bigskip \cite{GU1} Let $\mathbf{T}=\left( T_{1},\cdots
,T_{n}\right) $ $\in {\mathcal{B}}_{A^{{1}/{2}}}\left( {\mathcal{H}}%
\right) ^{n}$ be an $n$-tuple. Then%
\begin{equation*}
\frac{1}{2\sqrt{n}}\left\Vert \mathbf{T}\right\Vert _{A}\leq \omega
_{A}\left( \mathbf{T}\right) \leq \left\Vert \mathbf{T}\right\Vert _{A}\text{%
.}
\end{equation*}
\end{lemma}

Recently, Baklouti et al. \cite{BA2} have introduced the class of joint $A$-normality of operators on semi-Hilbertian spaces as follows:

\begin{definition}
An $n$-tuple $\mathbf{T=}\left( T_{1},...,T_{n}\right) \in {\mathcal{B}}%
_{A}\left( {\mathcal{H}}\right) ^{n}$ is said to be an $A$-normal $n$-tuple if $A\left[ T_{i},T_{j}\right] =0$ for all $i,j=
1,...,n $ and $T_{k}^{\sharp_A}T_{k}=T_{k}T_{k}^{\sharp_A}$ \ for all $k= 1,...,n $, where $\left[ T_{i},T_{j}\right]
=T_{i}T_{j}-T_{j}T_{i}$.
\end{definition}

We also recall here the commuting $n$-tuple operators.

\begin{definition}
Let $\mathbf{T=}\left( T_{1},...,T_{n}\right) \in {\mathcal{B}}\left( {
\mathcal{H}}\right) ^{n}$ be an $n$-tuple. Then $\mathbf{T}$ is said
to be commuting if $T_{i}T_{j}=T_{j}T_{i}$ for all $i,j=1,...,n$.
\end{definition}

\begin{lemma}
\label{L004}\cite{BA2} Let $\mathbf{T=}\left( T_{1},...,T_{n}\right) \in {%
\mathcal{B}}_{A^{{1}/{2}}}\left( {\mathcal{H}}\right) ^{n}$ be commuting and $A$-normal $n$-tuple. Then 
\begin{equation*}
\left\Vert \mathbf{T}\right\Vert _{A}=\omega _{A}\left( \mathbf{T}\right) 
\text{.}
\end{equation*}
\end{lemma}

\begin{lemma}
\label{L05}\bigskip \cite[Proposition 2.9]{GU1} Let $\ \mathbf{T=}\left(
T_{1},...,T_{n}\right) \in {\mathcal{B}}_{A}\left( {\mathcal{H}}\right) ^{n}$
be an $n$-tuple. Then
\begin{equation*}
\left\Vert \mathbf{T}\right\Vert _{A}=\left\Vert
\sum\limits_{k=1}^{n}T_{k}^{\sharp_A}T_{k}\right\Vert _{A}^{{1}/{2}}\text{.}
\end{equation*}
\end{lemma}

\begin{lemma}
\label{L5}\bigskip Let $\mathbf{T}=\left( T_{1},\cdots ,T_{n}\right) ,\, 
\mathbf{S}=\left( S_{1},\cdots ,S_{n}\right) $ $\in {\mathcal{B}}_{A^{1/2}}\left( {\mathcal{H}}\right) ^{n}$ be two $n$-tuples. Then
\begin{equation*}
\left\Vert \mathbf{TS}\right\Vert _{A}\leq \left\Vert \mathbf{T}\right\Vert
_{A}\left\Vert \mathbf{S}\right\Vert _{A}\text{.}
\end{equation*}
\end{lemma}

\begin{proof}
Let $x\in {\mathcal{H}}$ with $\left\Vert x\right\Vert _{A}=1$. For
all $k=1,2,...,n$, we see that
\begin{equation*}
\left\Vert T_{k}x\right\Vert _{A}\leq \left( \sum\limits_{k=1}^{n}\left\Vert
T_{k}x\right\Vert _{A}^{2}\right) ^{{1}/{2}}\text{.}
\end{equation*}
Taking the supremum over $x\in {\mathcal{H}}$ with $\left\Vert x\right\Vert
_{A}=1$, we get
$\left\Vert T_{k}\right\Vert _{A}\leq \left\Vert \mathbf{T}\right\Vert _{A}%
\text{ for all }k=1,2,...,n\text{.}$
Therefore, we have
\begin{eqnarray*}
\left\Vert \mathbf{TS}\right\Vert _{A} &=&\sup\limits_{\left\Vert
x\right\Vert _{A}=1}\left( \sum\limits_{k=1}^{n}\left\Vert
T_{k}S_{k}x\right\Vert _{A}^{2}\right) ^{{1}/{2}} 
\leq \sup\limits_{\left\Vert x\right\Vert _{A}=1}\left(
\sum\limits_{k=1}^{n}\left\Vert T_{k}\right\Vert _{A}^{2}\left\Vert
S_{k}x\right\Vert _{A}^{2}\right) ^{{1}/{2}} \\
&\leq &\sup\limits_{\left\Vert x\right\Vert _{A}=1}\left(
\sum\limits_{k=1}^{n}\left\Vert \mathbf{T}\right\Vert _{A}^{2}\left\Vert
S_{k}x\right\Vert _{A}^{2}\right) ^{{1}/{2}} 
=\left\Vert \mathbf{T}\right\Vert _{A}\sup\limits_{\left\Vert x\right\Vert
_{A}=1}\left( \sum\limits_{k=1}^{n}\left\Vert S_{k}x\right\Vert
_{A}^{2}\right) ^{{1}/{2}} \\
&=&\left\Vert \mathbf{T}\right\Vert _{A}\left\Vert \mathbf{S}\right\Vert _{A}%
\text{.}
\end{eqnarray*}
\end{proof}

\begin{lemma}
\label{L6}\bigskip\ Let $\mathbf{T}=\left( T_{1},\cdots ,T_{n}\right) ,\,
\mathbf{S}=\left( S_{1},\cdots ,S_{n}\right) $ $\in {\mathcal{B}}_{A^{1/2}}\left( {\mathcal{H}}\right) ^{n}$  be two $n$-tuples. Then 
\begin{equation*}
\omega _{A}\left( \mathbf{TS}\right) \leq 4n\omega _{A}\left( \mathbf{T}%
\right) \omega _{A}\left( \mathbf{S}\right) \text{.}
\end{equation*}
\end{lemma}

\begin{proof}
Following Lemmas \ref{L5} and \ref{L4}, we have 
\begin{equation*}
\omega _{A}\left( \mathbf{TS}\right) \leq \left\Vert \mathbf{TS}\right\Vert
_{A}\leq \left\Vert \mathbf{T}\right\Vert _{A}\left\Vert \mathbf{S}%
\right\Vert _{A}\leq 4n\omega _{A}\left( \mathbf{T}\right) \omega _{A}\left( 
\mathbf{S}\right) \text{.}
\end{equation*}
\end{proof}

 The proof of the following lemma can be followed similar to \cite[
Theorem 2.10]{J}. 

\begin{lemma}
\label{L66}\bigskip\ Let $\mathbf{T}=\left( T_{1},\cdots ,T_{n}\right) ,\,
\mathbf{S}=\left( S_{1},\cdots ,S_{n}\right) $ $\in {\mathcal{B}}_{A^{1/2}}\left( {\mathcal{H}}\right) ^{n}$. If $\mathbf{TS=ST}$, then 
\begin{equation*}
\omega _{A}\left( \mathbf{TS}\right) \leq 2\sqrt{n}\omega _{A}\left( \mathbf{%
T}\right) \omega _{A}\left( \mathbf{S}\right) \text{.}
\end{equation*}
\end{lemma}

\begin{lemma}
\label{L7}\bigskip Let $\mathbf{T}=\left( T_{1},\cdots ,T_{n}\right) \in {%
\mathcal{B}}_{A}\left( {\mathcal{H}}\right) ^{n}$ and $\mathbf{S}=\left(
S_{1},\cdots ,S_{n}\right) $ $\in {\mathcal{B}}_{A^{{1}/{2}}}\left( {%
\mathcal{H}}\right) ^{n}$. If $T_{k}$ is an $A$-isometry (that is, $
T_{k}^{\sharp _{A}}T_{k}=P$ for every $k=1,...,n)$, then
\end{lemma}
\begin{enumerate}
  
\item  $\omega _{A}\left( \mathbf{TS}\right) \leq \left\Vert \mathbf{S}%
\right\Vert _{A}$.

\item  $\left\Vert \mathbf{TS}\right\Vert _{A}\leq \left\Vert \mathbf{S}%
\right\Vert _{A}$.
  
\end{enumerate}

\begin{proof}
(1) By using the $A$-Cauchy--Schwarz inequality, we get%
\begin{eqnarray*}
\omega _{A}\left( \mathbf{TS}\right) &=&\sup_{\left\Vert x\right\Vert
_{A}=1}\left( \sum\limits_{k=1}^{n}\left\vert \left\langle
T_{k}S_{k}x,x\right\rangle _{A}\right\vert ^{2}\right) ^{{1}/{2}}\text{ }
\leq \sup_{\left\Vert x\right\Vert _{A}=1}\left(
\sum\limits_{k=1}^{n}\left\Vert T_{k}S_{k}x\right\Vert _{A}^{2}\right) ^{%
\frac{1}{2}}\text{ } \\
&=&\sup_{\left\Vert x\right\Vert _{A}=1}\left(
\sum\limits_{k=1}^{n}\left\langle T_{k}S_{k}x,T_{k}S_{k}x\right\rangle
_{A}\right) ^{{1}/{2}}\text{ } 
=\sup_{\left\Vert x\right\Vert _{A}=1}\left(
\sum\limits_{k=1}^{n}\left\langle T_{k}^{\sharp
_{A}}T_{k}S_{k}x,S_{k}x\right\rangle _{A}\right) ^{{1}/{2}} \\
&\leq&\text{ }\sup_{\left\Vert x\right\Vert _{A}=1}\left(
\sum\limits_{k=1}^{n}\left\langle S_{k}x,S_{k}x\right\rangle _{A}\right) ^{%
\frac{1}{2}} 
=\sup_{\left\Vert x\right\Vert _{A}=1}\left(
\sum\limits_{k=1}^{n}\left\Vert S_{k}x\right\Vert _{A}^{2}\right) ^{\frac{1}{%
2}} \\
&=&\left\Vert \mathbf{S}\right\Vert _{A}\text{.}
\end{eqnarray*}

(2) Similarly, the proof follows from the definition of $A$-joint operator
seminorm for $\mathbf{TS}$.
\end{proof}

\section{Main results}\label{sec3}

Here, we study our main results. We begin by introducing a new seminorm on ${\mathcal{B}}_{A^{{1}/{2}}}\left( {
\mathcal{H}}\right) ^{n}$.
\begin{definition}
    Let $\mathbf{T}=\left( T_{1},\cdots ,T_{n}\right) \in {\mathcal{B}}_{A^{
{1}/{2}}}\left( {\mathcal{H}}\right) ^{n}$ be an $n$-tuple
and let $\alpha ,\beta \geq 0$ such that $\left( \alpha ,\beta \right) \neq
\left( 0,0\right) $. Consider a mapping $\left\Vert \mathbf{\cdot }%
\right\Vert _{A_{\alpha ,\beta }}:{\mathcal{B}}_{A^{{1}/{2}}}\left( {%
\mathcal{H}}\right) ^{n}\rightarrow 
\mathbb{R}
^{+}$ defined as 
\begin{equation*}
\left\Vert \mathbf{T}\right\Vert _{A_{\alpha ,\beta }}:=\sup \left\{ \left(
\sum\limits_{k=1}^{n}\left( \alpha \left\vert \left\langle
T_{k}x,x\right\rangle _{A}\right\vert ^{2}+\beta \left\Vert
T_{k}x\right\Vert _{A}^{2}\right) \right) ^{{1}/{2}}:x\in {\mathcal{H}}%
\text{, }\left\Vert x\right\Vert _{A}=1\text{ }\right\} \text{.}
\end{equation*}
\end{definition}
We say that $\left\Vert \mathbf{T}\right\Vert _{A_{\alpha ,\beta }}$ is $\left( \alpha ,\beta \right) $-$A$-Euclidean operator radius of $\mathbf{T}$.
\begin{remark}
(i) For $\alpha =0$, $\beta =1$,  $\left\Vert \mathbf{T}\right\Vert
_{A_{\alpha ,\beta }}=\left\Vert \mathbf{T}\right\Vert _{A}$.
(ii) For $\alpha =1$, $\beta =0$, $\left\Vert \mathbf{T}\right\Vert
_{A_{\alpha ,\beta }}=\omega _{A}\left( \mathbf{T}\right) $.

\end{remark}

First, we show that $\left\Vert \mathbf{T}\right\Vert
_{A_{\alpha ,\beta }}$ defnes a seminorm on ${\mathcal{B}}_{A^{{1}/{2}
}}\left( {\mathcal{H}}\right) ^{n}$. 

\begin{proposition}
\label{P1}Let $\mathbf{T}=\left( T_{1},\cdots ,T_{n}\right) \in {\mathcal{B}}%
_{A^{{1}/{2}}}\left( {\mathcal{H}}\right) ^{n}$ and $\ \mathbf{S=}\left(
S_{1},\cdots ,S_{n}\right) \in {\mathcal{B}}_{A^{{1}/{2}}}\left( {%
\mathcal{H}}\right) ^{n}$. Then, the following properties hold:
\end{proposition}

(i) $\left\Vert \mathbf{T}\right\Vert _{_{A_{\alpha ,\beta }}}=0$ if and
only if $AT_{k}=0$ for all $k=1,...,n$.

(ii) $\left\Vert \lambda \mathbf{T}\right\Vert _{_{A_{\alpha ,\beta
}}}=\left\vert \lambda \right\vert \left\Vert \mathbf{T}\right\Vert
_{_{A_{\alpha ,\beta }}}$ for all $\lambda \in 
\mathbb{C}
$.

(iii) $\left\Vert \mathbf{T+S}\right\Vert _{_{A_{\alpha ,\beta }}}\leq
\left\Vert \mathbf{T}\right\Vert _{_{A_{\alpha ,\beta }}}+\left\Vert \mathbf{%
S}\right\Vert _{_{A_{\alpha ,\beta }}}$.

\begin{proof}
The proofs (i) and (ii) follow easily, we only prove (iii).
Let $x\in {\mathcal{H}}$ with $\left\Vert x\right\Vert _{A}=1$. Employing the Cauchy--Schwarz inequality for sums, we have
\begin{eqnarray*}
&&\sum\limits_{k=1}^{n}\left( \alpha \left\vert \left\langle \left(
T_{k}+S_{k}\right) x,x\right\rangle _{A}\right\vert ^{2}+\beta \left\Vert
\left( T_{k}+S_{k}\right) x\right\Vert _{A}^{2}\right) \\
&\leq &\sum\limits_{k=1}^{n}\left( \alpha \left( \left\vert \left\langle
T_{k}x,x\right\rangle _{A}\right\vert +\left\vert \left\langle
S_{k}x,x\right\rangle _{A}\right\vert \right) ^{2}+\beta \left( \left\Vert
T_{k}x\right\Vert _{A}+\left\Vert S_{k}x\right\Vert _{A}\right) ^{2}\right)
\\
&=&\sum\limits_{k=1}^{n}\left( \alpha \left\vert \left\langle
T_{k}x,x\right\rangle _{A}\right\vert ^{2}+\beta \left\Vert
T_{k}x\right\Vert _{A}^{2}\right) +\sum\limits_{k=1}^{n}\left( \alpha
\left\vert \left\langle S_{k}x,x\right\rangle _{A}\right\vert ^{2}+\beta
\left\Vert S_{k}x\right\Vert _{A}^{2}\right) \\
&&+2\sum\limits_{k=1}^{n}\left( \alpha \left\vert \left\langle
T_{k}x,x\right\rangle _{A}\right\vert \left\vert \left\langle
S_{k}x,x\right\rangle _{A}\right\vert +\beta \left\Vert T_{k}x\right\Vert
_{A}\left\Vert S_{k}x\right\Vert _{A}\right) \\
&=&\sum\limits_{k=1}^{n}\left( \alpha \left\vert \left\langle
T_{k}x,x\right\rangle _{A}\right\vert ^{2}+\beta \left\Vert
T_{k}x\right\Vert _{A}^{2}\right) +\sum\limits_{k=1}^{n}\left( \alpha
\left\vert \left\langle S_{k}x,x\right\rangle _{A}\right\vert ^{2}+\beta
\left\Vert S_{k}x\right\Vert _{A}^{2}\right) \\
&&+2\sum\limits_{k=1}^{n}\left( \left\vert \left\langle \sqrt{\alpha }%
T_{k}x,x\right\rangle _{A}\right\vert \left\vert \left\langle \sqrt{\alpha }%
S_{k}x,x\right\rangle _{A}\right\vert +\left\Vert \sqrt{\beta }%
T_{k}x\right\Vert _{A}\left\Vert \sqrt{\beta }S_{k}x\right\Vert _{A}\right)
\\
&\leq &\sum\limits_{k=1}^{n}\left( \alpha \left\vert \left\langle
T_{k}x,x\right\rangle _{A}\right\vert ^{2}+\beta \left\Vert
T_{k}x\right\Vert _{A}^{2}\right) +\sum\limits_{k=1}^{n}\left( \alpha
\left\vert \left\langle S_{k}x,x\right\rangle _{A}\right\vert ^{2}+\beta
\left\Vert S_{k}x\right\Vert _{A}^{2}\right) \\
&&+2\sum\limits_{k=1}^{n}\sqrt{\alpha \left\vert \left\langle
T_{k}x,x\right\rangle _{A}\right\vert ^{2}+\beta \left\Vert
T_{k}x\right\Vert _{A}^{2}}\sqrt{\alpha \left\vert \left\langle
S_{k}x,x\right\rangle _{A}\right\vert ^{2}+\beta \left\Vert
S_{k}x\right\Vert _{A}^{2}} \\
&\leq &\sum\limits_{k=1}^{n}\left( \alpha \left\vert \left\langle
T_{k}x,x\right\rangle _{A}\right\vert ^{2}+\beta \left\Vert
T_{k}x\right\Vert _{A}^{2}\right) +\sum\limits_{k=1}^{n}\left( \alpha
\left\vert \left\langle S_{k}x,x\right\rangle _{A}\right\vert ^{2}+\beta
\left\Vert S_{k}x\right\Vert _{A}^{2}\right) \\
&&+2\left( \sum\limits_{k=1}^{n}\left( \alpha \left\vert \left\langle
T_{k}x,x\right\rangle _{A}\right\vert ^{2}+\beta \left\Vert
T_{k}x\right\Vert _{A}^{2}\right) \right) ^{{1}/{2}}\left(
\sum\limits_{k=1}^{n}\left( \alpha \left\vert \left\langle
S_{k}x,x\right\rangle _{A}\right\vert ^{2}+\beta \left\Vert
S_{k}x\right\Vert _{A}^{2}\right) \right) ^{{1}/{2}} \\
&\leq &\left\Vert \mathbf{T}\right\Vert _{A_{\alpha ,\beta }}^{2}+\left\Vert 
\mathbf{S}\right\Vert _{A_{\alpha ,\beta }}^{2}+2\left\Vert \mathbf{T}%
\right\Vert _{A_{\alpha ,\beta }}\left\Vert \mathbf{S}\right\Vert
_{A_{\alpha ,\beta }} 
=\left( \left\Vert \mathbf{T}\right\Vert _{A_{\alpha ,\beta }}+\left\Vert 
\mathbf{S}\right\Vert _{A_{\alpha ,\beta }}\right) ^{2}\text{.}
\end{eqnarray*}
Taking the supremum over $x\in {\mathcal{H}}$ with $\left\Vert x\right\Vert
_{A}=1$,  we get
%
$\left\Vert \mathbf{T+S}\right\Vert _{A_{\alpha ,\beta }}\leq \left\Vert 
\mathbf{T}\right\Vert _{A_{\alpha ,\beta }}+\left\Vert \mathbf{S}\right\Vert
_{A_{\alpha ,\beta }}\text{.}$
\end{proof}

In particular, if we consider the $n$-tuple $\mathbf{T}=\left( S,\cdots ,S\right) \in {\mathcal{B}}_{A^{{1}/{2}}}\left( {\mathcal{H}}\right) ^{n}$, then we get

\begin{proposition}
\label{P2} \bigskip Let $S\in {\mathcal{B}}_{A}\left( {\mathcal{H}}\right) $
and $\mathbf{T}=\left( S,\cdots ,S\right) \in {
\mathcal{B}}_{A^{{1}/{2}}}\left( {\mathcal{H}}\right) ^{n}$. For $0\leq \alpha \leq 1$ and $\beta=1-\alpha$, we have
\begin{equation*}
\left\Vert \mathbf{T}\right\Vert _{A_{\alpha ,\beta }}=\sqrt{n}\left\Vert
S\right\Vert _{A_{\alpha}}\text{.}
\end{equation*}
\end{proposition}

Next we show that $\left\Vert \mathbf{\cdot }\right\Vert
_{A_{\alpha ,\beta }}$ is equivalent to the $A$-Euclidean operator radius
and the $A$-joint operator norm on ${\mathcal{B}}_{A^{{1}/{2}
}}\left( {\mathcal{H}}\right) ^{n}$. 

\begin{theorem}
\label{TH1} Let $\mathbf{T}=\left( T_{1},\cdots ,T_{n}\right) \in {\mathcal{B
}}_{A^{{1}/{2}}}\left( {\mathcal{H}}\right) ^{n}$. Then
\end{theorem}

(i) $\sqrt{\alpha +\beta } \, \omega _{A}\left( \mathbf{T}\right) \leq
\left\Vert \mathbf{T}\right\Vert _{A_{\alpha ,\beta }}\leq \sqrt{\alpha
+4\beta n} \,\omega _{A}\left( \mathbf{T}\right) $.

(ii) $\max \left\{ \sqrt{\beta },\frac{1}{2}\sqrt{\frac{\alpha +\beta }{n}}%
\right\} \left\Vert \mathbf{T}\right\Vert _{A}\leq \left\Vert \mathbf{T}%
\right\Vert _{A_{\alpha ,\beta }}\leq \sqrt{\alpha +\beta }\left\Vert 
\mathbf{T}\right\Vert _{A}$.

\begin{proof}
(i) Let $x\in {\mathcal{H}}$ with $\left\Vert x\right\Vert _{A}=1$. We have 
\begin{eqnarray*}
\left\Vert \mathbf{T}\right\Vert _{A_{\alpha ,\beta }}^{2}
&=&\sup_{\left\Vert x\right\Vert _{A}=1}\sum\limits_{k=1}^{n}\left( \alpha
\left\vert \left\langle T_{k}x,x\right\rangle _{A}\right\vert ^{2}+\beta
\left\Vert T_{k}x\right\Vert _{A}^{2}\right) \\
&= &\sup_{\left\Vert x\right\Vert _{A}=1}\left(
\sum\limits_{k=1}^{n}\alpha \left\vert \left\langle T_{k}x,x\right\rangle
_{A}\right\vert ^{2}+\sum\limits_{k=1}^{n}\beta \left\Vert T_{k}x\right\Vert
_{A}^{2}\right) \\
&\leq &\sup_{\left\Vert x\right\Vert _{A}=1}\left(
\sum\limits_{k=1}^{n}\alpha \left\vert \left\langle T_{k}x,x\right\rangle
_{A}\right\vert ^{2}\right) +\sup_{\left\Vert x\right\Vert =1}\left(
\sum\limits_{k=1}^{n}\beta \left\Vert T_{k}x\right\Vert _{A}^{2}\right) \\
&=&\alpha \omega _{A}^{2}\left( \mathbf{T}\right) +\beta \left\Vert \mathbf{T%
}\right\Vert _{A}^{2} \\
&\leq &\alpha \omega _{A}^{2}\left( \mathbf{T}\right) +4\beta n\omega
_{A}^{2}\left( \mathbf{T}\right) \quad
\text{(by Lemma \ref{L4})} \\
&=&\left( \alpha +4\beta n\right) \omega _{A}^{2}\left( \mathbf{T}\right) 
\text{.}
\end{eqnarray*}
Also, we have
\begin{eqnarray*}
\left\Vert \mathbf{T}\right\Vert _{A_{\alpha ,\beta }} &=&\sup_{\left\Vert
x\right\Vert _{A}=1}\left( \sum\limits_{k=1}^{n}\alpha \left\vert
\left\langle T_{k}x,x\right\rangle _{A}\right\vert ^{2}+\beta \left\Vert
T_{k}x\right\Vert _{A}^{2}\right) ^{{1}/{2}} \\
&\geq &\sup_{\left\Vert x\right\Vert _{A}=1}\left(
\sum\limits_{k=1}^{n}\alpha \left\vert \left\langle T_{k}x,x\right\rangle
_{A}\right\vert ^{2}+\beta \left\vert \left\langle T_{k}x,x\right\rangle
_{A}\right\vert ^{2}\right) ^{{1}/{2}} \\
&=&\sup_{\left\Vert x\right\Vert _{A}=1}\left( \sum\limits_{k=1}^{n}\left(
\alpha +\beta \right) \left\vert \left\langle T_{k}x,x\right\rangle
_{A}\right\vert ^{2}\right) ^{{1}/{2}} \\
&=&\sqrt{\alpha +\beta }\,\omega _{A}\left( \mathbf{T}\right) \text{.}
\end{eqnarray*}
Therefore,
\begin{equation*}
\sqrt{\alpha +\beta }\,\omega _{A}\left( \mathbf{T}\right) \leq \left\Vert 
\mathbf{T}\right\Vert _{A_{\alpha ,\beta }}\leq \sqrt{\alpha +4\beta n}\,
\omega _{A}\left( \mathbf{T}\right) \text{.}
\end{equation*}

(ii) Let $x\in {\mathcal{H}}$ with $\left\Vert x\right\Vert _{A}=1$. We have%
\begin{eqnarray*}
\left\Vert \mathbf{T}\right\Vert _{A_{\alpha ,\beta }} &=&\sup_{\left\Vert
x\right\Vert _{A}=1}\left( \sum\limits_{k=1}^{n}\left( \alpha \left\vert
\left\langle T_{k}x,x\right\rangle _{A}\right\vert ^{2}+\beta \left\Vert
T_{k}x\right\Vert _{A}^{2}\right) \right) ^{{1}/{2}} \\
&\leq &\sup_{\left\Vert x\right\Vert _{A}=1}\left(
\sum\limits_{k=1}^{n}\left( \alpha \left\Vert T_{k}x\right\Vert
_{A}^{2}+\beta \left\Vert T_{k}x\right\Vert _{A}^{2}\right) \right) ^{\frac{1%
}{2}} \\
&=&\sup_{\left\Vert x\right\Vert _{A}=1}\left( \sum\limits_{k=1}^{n}\left(
\alpha +\beta \right) \left\Vert T_{k}x\right\Vert _{A}^{2}\right) ^{\frac{1%
}{2}} 
=\sqrt{\alpha +\beta }\left\Vert \mathbf{T}\right\Vert _{A}\text{.}
\end{eqnarray*}
Moreover, we have
\begin{eqnarray*}
\left\Vert \mathbf{T}\right\Vert _{A_{\alpha ,\beta }} &\geq &\sqrt{\alpha
+\beta }\omega _{A}\left( \mathbf{T}\right) 
\geq \frac{1}{2\sqrt{n}}\sqrt{\alpha +\beta }\left\Vert \mathbf{T}%
\right\Vert _{A} 
=\frac{1}{2}\sqrt{\frac{\alpha +\beta }{n}}\left\Vert \mathbf{T}%
\right\Vert _{A}\text{}
\end{eqnarray*}
and
$\left\Vert \mathbf{T}\right\Vert _{A_{\alpha ,\beta }}\geq \sqrt{\beta }%
\left\Vert \mathbf{T}\right\Vert _{A}\text{.} $ 
Thus, 
$\max \left\{ \sqrt{\beta },\frac{1}{2}\sqrt{\frac{\alpha +\beta }{n}}%
\right\} \left\Vert \mathbf{T}\right\Vert _{A}\leq \left\Vert \mathbf{T}%
\right\Vert _{A_{\alpha ,\beta }}\text{.}$
Hence,%
\begin{equation*}
\max \left\{ \sqrt{\beta }, \, \frac{1}{2}\sqrt{\frac{\alpha +\beta }{n}}
\right\} \left\Vert \mathbf{T}\right\Vert _{A}\leq \left\Vert \mathbf{T}
\right\Vert _{A_{\alpha ,\beta }}\leq \sqrt{\alpha +\beta }\left\Vert 
\mathbf{T}\right\Vert _{A}\text{.}
\end{equation*}
\end{proof}

\begin{remark}
\bigskip By setting $n=1$ and $A=I$ in Theorem \ref{TH1}, we can also obtain \cite[%
Theorem 2.1]{Sain}.
\end{remark}

In the following proposition, we prove that the $\left(\alpha ,\beta \right)$-$A$-Euclidean operator radius is weakly $A$-unitarily invariant.

\begin{proposition}
Let $\mathbf{T}=\left( T_{1},\cdots ,T_{n}\right) \in {\mathcal{B}}_A^{}\left( {\mathcal{H}}\right) ^{n}$ be an $n$-tuple 
and let $U\in {\mathcal{B}}_{A}\left( {\mathcal{H}}\right) $ be an $A$-unitary (i.e., $\left\Vert Ux\right\Vert
_{A}=\left\Vert U^{\sharp _{A}}x\right\Vert _{A}=\left\Vert x\right\Vert
_{A} $ for all $x\in {\mathcal{H}}$).  Then 
\begin{equation*}
\left\Vert U\mathbf{T}U^{\sharp _{A}}\right\Vert _{\alpha ,\beta
}=\left\Vert \mathbf{T}\right\Vert _{\alpha ,\beta }, \quad \text{where  $U\mathbf{T}U^{\sharp _{A}}:=\left( UT_{1}U^{\sharp _{A}},\cdots
,UT_{n}U^{\sharp _{A}}\right) $.}
\end{equation*}
\end{proposition}

\begin{proof}
Since $U$ is an $A$-unitary operator, it follows that 
\begin{eqnarray*}
\left\Vert U\mathbf{T}U^{\sharp _{A}}\right\Vert _{A_{\alpha ,\beta }}
&=&\sup_{\left\Vert x\right\Vert _{A}=1}\left( \sum\limits_{k=1}^{n} \left(\alpha
\left\vert \left\langle UT_{k}U^{\sharp _{A}}x,x\right\rangle
_{A}\right\vert ^{2}+\beta \left\Vert UT_{k}U^{\sharp _{A}}x\right\Vert
_{A}^{2} \right)\right) ^{{1}/{2}} \\
&=&\sup_{\left\Vert x\right\Vert _{A}=1}\left( \sum\limits_{k=1}^{n} \left( \alpha
\left\vert \left\langle T_{k}U^{\sharp _{A}}x,U^{\sharp _{A}}x\right\rangle _{A}\right\vert ^{2}+\beta
\left\Vert UT_{k}U^{\sharp _{A}}x\right\Vert _{A}^{2}\right)\right) ^{{1}/{2}}
\\
&=&\sup_{\left\Vert y\right\Vert _{A}=1}\left( \sum\limits_{k=1}^{n} \left( \alpha
\left\vert \left\langle T_{k}y,y\right\rangle _{A}\right\vert ^{2}+\beta
\left\Vert T_{k}y\right\Vert _{A}^{2}\right)\right) ^{{1}/{2}} 
\quad \text{(setting $y =Ux$)} \\
&\leq&\left\Vert \mathbf{T}\right\Vert _{\alpha ,\beta }.
\end{eqnarray*}
We now show the reverse inequality. 
Since $U$ is an $A$-unitary, we have $U^{\sharp_A}U=P$ and so
\begin{eqnarray*}
    \left\Vert \mathbf{T}\right\Vert _{\alpha ,\beta } &=&\sup_{\left\Vert y\right\Vert _{A}=1}\left( \sum\limits_{k=1}^{n} \left( \alpha
\left\vert \left\langle T_{k}y,y\right\rangle _{A}\right\vert ^{2}+\beta
\left\Vert T_{k}y\right\Vert _{A}^{2}\right)\right) ^{{1}/{2}} \\
&=&\sup_{\left\Vert y\right\Vert _{A}=1}\left( \sum\limits_{k=1}^{n} \left( \alpha
\left\vert \left\langle T_{k}y,y\right\rangle _{A}\right\vert ^{2}+\beta
\left\Vert UT_{k}y\right\Vert _{A}^{2}\right)\right) ^{{1}/{2}} \\
&=&\sup_{\left\Vert Py\right\Vert _{A}=1}\left( \sum\limits_{k=1}^{n} \left( \alpha
\left\vert \left\langle T_{k}Py,Py\right\rangle _{A}\right\vert ^{2}+\beta
\left\Vert UT_{k}Py\right\Vert _{A}^{2}\right)\right) ^{{1}/{2}} \,\, (\text{set } y=Py+x, \, Ax=0) \\
&=&\sup_{\left\Vert Uy\right\Vert _{A}=1}\left( \sum\limits_{k=1}^{n} \left( \alpha
\left\vert \left\langle UT_{k}U^{\sharp_A}Uy,Uy\right\rangle _{A}\right\vert ^{2}+\beta
\left\Vert UT_{k}U^{\sharp_A} Uy\right\Vert _{A}^{2}\right)\right) ^{{1}/{2}}\\
&\leq& \left\Vert U\mathbf{T}U^{\sharp _{A}}\right\Vert _{A_{\alpha ,\beta }}.
\end{eqnarray*}
\end{proof}

Next, we obtain an inequality for the product of two $n$-tuples.

\begin{theorem}
\label{TH2} Let $\mathbf{T}=\left( T_{1},\cdots ,T_{n}\right) ,\,  \mathbf{S=}\left(
S_{1},\cdots ,S_{n}\right) \in {\mathcal{B}}_{A^{{1}/{2}}}\left( {
\mathcal{H}}\right) ^{n}$ and let $\beta \neq 0$. Then
\begin{equation*}
\left\Vert \mathbf{TS}\right\Vert _{A_{\alpha ,\beta }}\leq \min \left\{ 2%
\sqrt{\frac{n}{\beta }},\frac{\sqrt{\alpha +\beta }}{\beta },\frac{4n}{\sqrt{%
\alpha +\beta }}\right\} \left\Vert \mathbf{T}\right\Vert _{A_{\alpha ,\beta
}}\left\Vert \mathbf{S}\right\Vert _{A_{\alpha ,\beta }}\text{.}
\end{equation*}
\end{theorem}

\begin{proof}
Following Theorem \ref{TH1}, and using Lemmas \ref{L5} and \ref{L4}, we have
\begin{eqnarray*}
\left\Vert \mathbf{TS}\right\Vert _{A_{\alpha ,\beta }} \leq \sqrt{\alpha
+\beta }\left\Vert \mathbf{TS}\right\Vert _{A} 
\leq \sqrt{\alpha +\beta }\left\Vert \mathbf{T}\right\Vert _{A}\left\Vert 
\mathbf{S}\right\Vert _{A}
&\leq & 2\sqrt{n}\sqrt{\alpha +\beta }\omega _{A}\left( \mathbf{T}\right)
\left\Vert \mathbf{S}\right\Vert _{A} \\
&\leq & 2\sqrt{\frac{n}{\beta }}\left\Vert \mathbf{T}\right\Vert _{A_{\alpha
,\beta }}\left\Vert \mathbf{S}\right\Vert _{A_{\alpha ,\beta }}\text{.}
\end{eqnarray*}
Since $\sqrt{\beta }\left\Vert \mathbf{T}\right\Vert _{A}\leq \left\Vert 
\mathbf{T}\right\Vert _{A_{\alpha ,\beta }}$ and $\sqrt{\beta }\left\Vert 
\mathbf{S}\right\Vert _{A}\leq \left\Vert \mathbf{S}\right\Vert _{A_{\alpha
,\beta }}$, similarly we also have  
\begin{eqnarray*}
\left\Vert \mathbf{TS}\right\Vert _{A_{\alpha ,\beta }} \leq \sqrt{\alpha
+\beta }\left\Vert \mathbf{TS}\right\Vert _{A} 
\leq \sqrt{\alpha +\beta }\left\Vert \mathbf{T}\right\Vert _{A}\left\Vert 
\mathbf{S}\right\Vert _{A}
\leq \frac{\sqrt{\alpha +\beta }}{\beta }\left\Vert \mathbf{T}\right\Vert
_{A_{\alpha ,\beta }}\left\Vert \mathbf{S}\right\Vert _{A_{\alpha ,\beta }}%
\text{.}
\end{eqnarray*}
Again, since $\left\Vert \mathbf{T}\right\Vert _{A}\leq 2\sqrt{n}\omega
_{A}\left( \mathbf{T}\right) $ and $\left\Vert \mathbf{S}\right\Vert
_{A}\leq 2\sqrt{n}\omega _{A}\left( \mathbf{S}\right) $, we have
\begin{eqnarray*}
\left\Vert \mathbf{TS}\right\Vert _{A_{\alpha ,\beta }} \leq \sqrt{\alpha
+\beta }\left\Vert \mathbf{TS}\right\Vert _{A} 
\leq \sqrt{\alpha +\beta }\left\Vert \mathbf{T}\right\Vert _{A}\left\Vert 
\mathbf{S}\right\Vert _{A} 
&\leq &4n\sqrt{\alpha +\beta }\omega _{A}\left( \mathbf{T}\right) \omega
_{A}\left( \mathbf{S}\right) \\
&\leq &\frac{4n}{\sqrt{\alpha +\beta }}\left\Vert \mathbf{T}\right\Vert
_{A_{\alpha ,\beta }}\left\Vert \mathbf{S}\right\Vert _{A_{\alpha ,\beta }}%
\text{.}
\end{eqnarray*}
Combining above inequalities, we get the desired
inequality.
\end{proof}

\begin{remark}
\bigskip By setting $n=1$ in Theorem \ref{TH2}, we obtain
\begin{equation*}
\left\Vert TS\right\Vert _{A_{\alpha ,\beta }}\leq \min \left\{ 2\sqrt{\frac{%
1}{\beta }},\frac{\sqrt{\alpha +\beta }}{\beta },\frac{4}{\sqrt{\alpha
+\beta }}\right\} \left\Vert T\right\Vert _{A_{\alpha ,\beta }}\left\Vert
S\right\Vert _{A_{\alpha ,\beta }}\text{,}
\end{equation*}
where \begin{equation*}
\left\Vert T\right\Vert _{A_{\alpha,\beta }}:=\sup \left\{  \left( \alpha
\left\vert \left\langle Tx,x\right\rangle _{A}\right\vert ^{2}+\beta \left\Vert Tx\right\Vert _{A}^{2} \right) ^{\frac{1}{%
2}}:x\in {\mathcal{H}}\text{, }\left\Vert x\right\Vert _{A}=1\text{ }%
\right\} \text{.} 
\end{equation*}
From the above inequality, for $A=I$, we get the inequality \cite[Theorem 2.9]{Sain}, namely
\begin{equation*}
\left\Vert TS\right\Vert _{_{\alpha ,\beta }}\leq \min \left\{ 2\sqrt{\frac{1%
}{\beta }},\frac{\sqrt{\alpha +\beta }}{\beta },\frac{4}{\sqrt{\alpha +\beta 
}}\right\} \left\Vert T\right\Vert _{_{\alpha ,\beta }}\left\Vert
S\right\Vert _{_{\alpha ,\beta }}\text{.}
\end{equation*}
\end{remark}

We again obtain inequalities for the product of two $n$-tuples under some assumptions.

\begin{theorem}
Let $\mathbf{T}=\left( T_{1},\cdots ,T_{n}\right) \in {\mathcal{B}}%
_{A}\left( {\mathcal{H}}\right) ^{n}$ and $\mathbf{S=}\left( S_{1},\cdots
,S_{n}\right) \in {\mathcal{B}}_{A^{{1}/{2}}}\left( {\mathcal{H}}\right)
^{n}$. Then the following inequalities hold:

(i) If $\mathbf{TS=ST}$ and $\beta \neq 0$, then 
\begin{equation*}
\left\Vert \mathbf{TS}\right\Vert _{A_{\alpha ,\beta }}\leq \sqrt{\frac{%
4\alpha n}{\left( \alpha +\beta \right) ^{2}}+\frac{1}{\beta }}\left\Vert 
\mathbf{T}\right\Vert _{A_{\alpha ,\beta }}\left\Vert \mathbf{S}\right\Vert
_{A_{\alpha ,\beta }}\text{.}
\end{equation*}

(ii) If $T_{k}$ is an $A$-isometry for each $k=1,2,...,n$, then%
\begin{equation*}
\left\Vert \mathbf{TS}\right\Vert _{A_{\alpha ,\beta }}\leq \sqrt{\frac{%
4\alpha n}{\alpha +\beta }+1}\left\Vert \mathbf{S}\right\Vert _{A_{\alpha
,\beta }}\text{.}
\end{equation*}
\end{theorem}

\begin{proof}
(i) Let $x\in {\mathcal{H}}$ with $\left\Vert x\right\Vert _{A}=1$. Then, we have
\begin{eqnarray*}
\left\Vert \mathbf{TS}\right\Vert _{A_{\alpha ,\beta }}^{2}
&=&\sup_{\left\Vert x\right\Vert _{A}=1}\sum\limits_{k=1}^{n}\left( \alpha
\left\vert \left\langle T_{k}S_{k}x,x\right\rangle _{A}\right\vert
^{2}+\beta \left\Vert T_{k}S_{k}x\right\Vert _{A}^{2}\right) \\
&\leq &\alpha \omega _{A}^{2}\left( \mathbf{TS}\right) +\beta \left\Vert 
\mathbf{TS}\right\Vert _{A}^{2} \\
&\leq &4\alpha n\omega _{A}^{2}\left( \mathbf{T}\right) \omega
_{A}^{2}\left( \mathbf{S}\right) +\beta \left\Vert \mathbf{T}\right\Vert
_{A}^{2}\left\Vert \mathbf{S}\right\Vert _{A}^{2} \quad
\text{(by Lemma \ref{L66})} \\
&\leq &\frac{4\alpha n}{\left( \alpha +\beta \right) ^{2}}\left\Vert \mathbf{%
T}\right\Vert _{A_{\alpha ,\beta }}^{2}\left\Vert \mathbf{S}\right\Vert
_{A_{\alpha ,\beta }}^{2}+\frac{1}{\beta }\left\Vert \mathbf{T}\right\Vert
_{A_{\alpha ,\beta }}^{2}\left\Vert \mathbf{S}\right\Vert _{A_{\alpha ,\beta
}}^{2} \\
&=&\left( \frac{4\alpha n}{\left( \alpha +\beta \right) ^{2}}+\frac{1}{\beta 
}\right) \left\Vert \mathbf{T}\right\Vert _{A_{\alpha ,\beta
}}^{2}\left\Vert \mathbf{S}\right\Vert _{A_{\alpha ,\beta }}^{2}\text{.}
\end{eqnarray*}
(ii) We also have
\begin{eqnarray*}
\left\Vert \mathbf{TS}\right\Vert _{A_{\alpha ,\beta }}^{2}
&\leq &\alpha \omega _{A}^{2}\left( \mathbf{TS}\right) +\beta \left\Vert 
\mathbf{TS}\right\Vert _{A}^{2} 
\leq \alpha \left\Vert \mathbf{S}\right\Vert _{A}^{2}+\beta \left\Vert 
\mathbf{S}\right\Vert _{A}^{2} \quad
\text{(by Lemma \ref{L7})} \\
&\leq &4\alpha n\omega _{A}^{2}\left( \mathbf{S}\right) +\beta \left\Vert 
\mathbf{S}\right\Vert _{A}^{2} 
\leq \frac{4\alpha n}{\alpha +\beta }\left\Vert \mathbf{S}\right\Vert
_{A_{\alpha ,\beta }}^{2}+\left\Vert \mathbf{S}\right\Vert _{A_{\alpha
,\beta }}^{2}\\ 
&=& \left( \frac{4\alpha n}{\alpha +\beta }+1\right) \left\Vert \mathbf{S}%
\right\Vert _{A_{\alpha ,\beta }}^{2}\text{.}
\end{eqnarray*}
\end{proof}

Next, we obtain a lower bound that
extends \cite[Theorem 2.7]{Sain}.

\begin{theorem}
Let $\mathbf{T}=\left( T_{1},\cdots ,T_{n}\right) \in {\mathcal{B}}%
_{A}\left( {\mathcal{H}}\right) ^{n}$ be an $n$-tuple. Then
\begin{equation*}
\left\Vert \mathbf{T}\right\Vert _{A_{\alpha ,\beta }}\geq \max \left\{ 
\sqrt{\alpha \omega _{A}^{2}\left( \mathbf{T}\right) +\beta m_{A}^{2}\left( 
\mathbf{T}\right) },\sqrt{\alpha c_{A}^{2}\left( 
\mathbf{T}\right) +\beta \left\Vert \mathbf{T}\right\Vert _{A}^{2}}\right\} 
\text{,}
\end{equation*}
where $m_{A}\left(
\mathbf{T}\right)= \inf  \left\{ \left( \sum\limits_{k=1}^{n} \left\langle T_{k}^{\sharp
_{A}}T_{k}x,x\right\rangle _{A} \right)^{1/2}: x\in \mathcal{H},\, \|x\|_A=1  \right\}$.
\end{theorem}

\begin{proof}
Let $x\in {\mathcal{H}}$ with $\left\Vert x\right\Vert _{A}=1$. Then, we have%
\begin{eqnarray*}
\left\Vert \mathbf{T}\right\Vert _{A_{\alpha ,\beta }}^{2} &\geq
&\sum\limits_{k=1}^{n}\left( \alpha \left\vert \left\langle
T_{k}x,x\right\rangle _{A}\right\vert ^{2}+\beta \left\Vert
T_{k}x\right\Vert _{A}^{2}\right) 
=\sum\limits_{k=1}^{n}\alpha \left\vert \left\langle T_{k}x,x\right\rangle
_{A}\right\vert ^{2}+\sum\limits_{k=1}^{n}\beta \left\langle T_{k}^{\sharp
_{A}}T_{k}x,x\right\rangle _{A} \\
&\geq &\alpha \sum\limits_{k=1}^{n}\left\vert \left\langle
T_{k}x,x\right\rangle _{A}\right\vert ^{2}+ \beta m_{A}^2\left(
\mathbf{T}\right).
\end{eqnarray*}
Taking the supremum over $\left\Vert
x\right\Vert =1$, we get
$\left\Vert \mathbf{T}\right\Vert _{A_{\alpha ,\beta }}^{2}\geq \alpha \omega
_{A}^{2}\left( \mathbf{T}\right) +\beta m_{A}^{2}\left( \mathbf{T}\right) \text{.} $ 
Also, we have 
\begin{eqnarray*}
\left\Vert \mathbf{T}\right\Vert _{A_{\alpha ,\beta }}^{2} &\geq
&\sum\limits_{k=1}^{n}\left( \alpha \left\vert \left\langle
T_{k}x,x\right\rangle _{A}\right\vert ^{2}+\beta \left\Vert
T_{k}x\right\Vert _{A}^{2}\right) 
=\alpha \sum\limits_{k=1}^{n}\left\vert \left\langle T_{k}x,x\right\rangle
_{A}\right\vert ^{2}+\beta \sum\limits_{k=1}^{n}\left\Vert T_{k}x\right\Vert
_{A}^{2} \\
&\geq &\alpha c_{A}^{2}\left( \mathbf{T}\right) +\beta
\sum\limits_{k=1}^{n}\left\Vert T_{k}x\right\Vert _{A}^{2}\text{.}
\end{eqnarray*}
Taking the supremum over $\left\Vert
x\right\Vert _{A}=1$, we get
$\left\Vert \mathbf{T}\right\Vert _{A_{\alpha ,\beta }}^{2}\geq \alpha
c_{A}^{2}\left( \mathbf{T}\right) +\beta \left\Vert \mathbf{T}\right\Vert
_{A}^{2}\text{.}$  
Thus, we conclude 
\begin{equation*}
\left\Vert \mathbf{T}\right\Vert _{A_{\alpha ,\beta }}\geq \max \left\{ 
\sqrt{\alpha \omega _{A}^{2}\left( \mathbf{T}\right) +\beta m_{A}^{2}\left( 
\mathbf{T} \right) },\sqrt{\alpha c_{A}^{2}\left( 
\mathbf{T}\right) +\beta \left\Vert \mathbf{T}\right\Vert _{A}^{2}}\right\} .
\end{equation*}
\end{proof}

Next, using the Cartesian decomposition of an operator, we obtain the following inequality.

\begin{theorem}\label{p--1}
\bigskip Let $\mathbf{T}=\left( T_{1},\cdots ,T_{n}\right) \in {\mathcal{B}}
_{A}\left( {\mathcal{H}}\right) ^{n}$ be a commuting $n$-tuple. Then
\begin{equation*}
\left\Vert \mathbf{T}\right\Vert _{A_{\alpha ,\beta }}^{2}\geq \left\Vert
\sum\limits_{k=1}^{n}\left( \frac{\alpha }{8}T_{k}T_{k}^{\sharp _{A}}+\left( 
\frac{\alpha }{8}+\frac{\beta }{2}\right) T_{k}^{\sharp _{A}}T_{k}\right)
\right\Vert _{A}.
\end{equation*}
\end{theorem}

\begin{proof}
Let $x\in {\mathcal{H}}$ and $\left\Vert x\right\Vert
_{A}=1$. Then, we see that
\begin{eqnarray*}
\left\vert \left\langle T_{k}x,x\right\rangle _{A}\right\vert ^{2}
&=&\left\vert \left\langle T_{k}^{\sharp _{A}}x,x\right\rangle
_{A}\right\vert ^{2} \\
&=&\left\vert \left\langle \left( \Re _{A}\left( T_{k}\right) \right)
^{\sharp _{A}}x,x\right\rangle _{A}\right\vert ^{2}+\left\vert \left\langle
\left( \Im _{A}\left( T_{k}\right) \right) ^{\sharp _{A}}x,x\right\rangle
_{A}\right\vert ^{2} \\
&\geq &\frac{1}{2}\left\vert \left\langle \left( \left( \Re _{A}\left(
T_{k}\right) \right) ^{\sharp _{A}}\pm \left( \Im _{A}\left( T_{k}\right)
\right) ^{\sharp _{A}}\right) x,x\right\rangle _{A}\right\vert ^{2}\text{.}
\end{eqnarray*}
Therefore, we have
\begin{eqnarray*}
\left\Vert \mathbf{T}\right\Vert _{A_{\alpha ,\beta }}^{2}
&=&\sup_{\left\Vert x\right\Vert _{A}=1}\sum\limits_{k=1}^{n}\left( \alpha
\left\vert \left\langle T_{k}x,x\right\rangle _{A}\right\vert ^{2}+\beta
\left\Vert T_{k}x\right\Vert _{A}^{2}\right) \\
&\geq &\sup_{\left\Vert x\right\Vert _{A}=1}\sum\limits_{k=1}^{n}\alpha
\left\vert \left\langle T_{k}x,x\right\rangle _{A}\right\vert ^{2} \\
&\geq &\frac{\alpha }{2}\sup_{\left\Vert x\right\Vert
_{A}=1}\sum\limits_{k=1}^{n}\left\vert \left\langle \left( \left( \Re
_{A}\left( T_{k}\right) \right) ^{\sharp _{A}}\pm \left( \Im _{A}\left(
T_{k}\right) \right) ^{\sharp _{A}}\right) x,x\right\rangle _{A}\right\vert
^{2} \\
&=&\frac{\alpha }{2}\omega _{A}^{2}\left( \mathbf{S}\right) \text{,}
\end{eqnarray*}
where $\mathbf{S=}\left( \left( \Re _{A}\left( T_{1}\right) \right) ^{\sharp
_{A}}\pm \left( \Im _{A}\left( T_{1}\right) \right) ^{\sharp
_{A}},...,\left( \Re _{A}\left( T_{n}\right) \right) ^{\sharp _{A}}\pm
\left( \Im _{A}\left( T_{n}\right) \right) ^{\sharp _{A}}\right) $. Since $
\mathbf{T}$ is a commuting $n$-tuple, it can be verified
that $\mathbf{S}$ is also commuting $n$-tuple. Moreover, since $
\left( \left( \Re _{A}\left( T_{k}\right) \right) ^{\sharp _{A}}\right)
^{\sharp _{A}}=\left( \Re _{A}\left( T_{k}\right) \right) ^{\sharp _{A}}$
and $\left( \left( \Im _{A}\left( T_{k}\right) \right) ^{\sharp _{A}}\right)
^{\sharp _{A}}=\left( \Im _{A}\left( T_{k}\right) \right) ^{\sharp _{A}}$, it follows that $S_{k}^{\sharp _{A}}=S_{k}$ for all $k=1,2,...,n$,
where $S_{k}:=\left( \Re _{A}\left( T_{k}\right) \right) ^{\sharp _{A}}\pm
\left( \Im _{A}\left( T_{k}\right) \right) ^{\sharp _{A}}$. This implies
that $\mathbf{S}$ is an $A$-normal commuting $n$-tuple. Therefore,
by using Lemma \ref{L004} and Lemma \ref{L05}, we get 
\begin{equation*}
\omega _{A}\left( \mathbf{S}\right) =\left\Vert S\right\Vert _{A}=\sqrt{%
\left\Vert \sum\limits_{k=1}^{n}S_{k}^{\sharp _{A}}S_{k}\right\Vert _{A}}=%
\sqrt{\left\Vert \sum\limits_{k=1}^{n}S_{k}^{2}\right\Vert _{A}}\text{.}
\end{equation*}
Hence, we get
\begin{eqnarray*}
\left\Vert \mathbf{T}\right\Vert _{A_{\alpha ,\beta }}^{2} &\geq &\frac{%
\alpha }{2}\omega _{A}^{2}\left( \mathbf{S}\right) =\frac{\alpha }{2}%
\left\Vert \sum\limits_{k=1}^{n}S_{k}^{2}\right\Vert _{A} 
=\frac{\alpha }{2}\sup_{\left\Vert x\right\Vert _{A}=1}\left\Vert
\sum\limits_{k=1}^{n}S_{k}^{2}x\right\Vert _{A} \\
&\geq &\frac{\alpha }{2}\sup_{\left\Vert x\right\Vert _{A}=1}\left\langle
\sum\limits_{k=1}^{n}\left( \left( \Re _{A}\left( T_{k}\right) \right)
^{\sharp _{A}}\pm \left( \Im _{A}\left( T_{k}\right) \right) ^{\sharp
_{A}}\right) ^{2}x,x\right\rangle _{A} \\
&\geq &\frac{\alpha }{4}\sup_{\left\Vert x\right\Vert _{A}=1}\left\langle
\sum\limits_{k=1}^{n}\left( \left( \Re _{A}\left( T_{k}\right) \right)
^{\sharp _{A}}+\left( \Im _{A}\left( T_{k}\right) \right) ^{\sharp
_{A}}\right) ^{2}+\left( \left( \Re _{A}\left( T_{k}\right) \right) ^{\sharp
_{A}}-\left( \Im _{A}\left( T_{k}\right) \right) ^{\sharp _{A}}\right)
^{2}x,x\right\rangle _{A} \\
&\geq &\frac{\alpha }{2}\sup_{\left\Vert x\right\Vert _{A}=1}\left\langle
\sum\limits_{k=1}^{n}\left( \left( \left( \Re _{A}\left( T_{k}\right)
\right) ^{\sharp _{A}}\right) ^{2}+\left( \left( \Im _{A}\left( T_{k}\right)
\right) ^{\sharp _{A}}\right) ^{2}\right) x,x\right\rangle _{A}\text{.}
\end{eqnarray*}
Also, we have
\begin{eqnarray*}
\left\Vert \mathbf{T}\right\Vert _{A_{\alpha ,\beta }}^{2} &\geq
&\sup_{\left\Vert x\right\Vert _{A}=1}\sum\limits_{k=1}^{n}\beta \left\Vert
T_{k}x\right\Vert _{A}^{2} 
=\sup_{\left\Vert x\right\Vert _{A}=1}\sum\limits_{k=1}^{n}\left\langle
\beta T^{\sharp _{A}}Tx,x\right\rangle _{A}\text{.}
\end{eqnarray*}
Since
\begin{equation*}
\left( \left( \Re _{A}\left( T_{k}\right) \right) ^{\sharp _{A}}\right)
^{2}+\left( \left( \Im _{A}\left( T_{k}\right) \right) ^{\sharp _{A}}\right)
^{2}=\left( \frac{TT^{\sharp _{A}}+T^{\sharp _{A}}T}{2}\right) ^{\sharp _{A}}%
\text{,}
\end{equation*}
and using the above inequalities, we get 
\begin{eqnarray*}
2\left\Vert \mathbf{T}\right\Vert _{A_{\alpha ,\beta }}^{2} &\geq
&\sup_{\left\Vert x\right\Vert _{A}=1}\left\langle
\sum\limits_{k=1}^{n}\left( \frac{\alpha }{2}\left( \left( \left( \Re
_{A}\left( T_{k}\right) \right) ^{\sharp _{A}}\right) ^{2}+\left( \left( \Im
_{A}\left( T_{k}\right) \right) ^{\sharp _{A}}\right) ^{2}\right) +\beta
T_{k}^{\sharp _{A}}T_{k}\right) x,x\right\rangle _{A} \\
&=&\left\Vert \sum\limits_{k=1}^{n}\left( \frac{\alpha }{2}\left( \left(
\left( \Re _{A}\left( T_{k}\right) \right) ^{\sharp _{A}}\right) ^{2}+\left(
\left( \Im _{A}\left( T_{k}\right) \right) ^{\sharp _{A}}\right) ^{2}\right)
+\beta T_{k}^{\sharp _{A}}T_{k}\right) \right\Vert _{A}\text{.} \\
&=&\left\Vert \sum\limits_{k=1}^{n}\left( \frac{\alpha }{2}\left( \frac{%
TT^{\sharp _{A}}+T^{\sharp _{A}}T}{2}\right) ^{\sharp _{A}}+\beta
T_{k}^{\sharp _{A}}T_{k}\right) \right\Vert _{A} \\
&=&\left\Vert \sum\limits_{k=1}^{n}\left( \frac{\alpha }{2}\left( \frac{%
TT^{\sharp _{A}}+T^{\sharp _{A}}T}{2}\right) ^{\sharp _{A}}+\beta \left(
T_{k}^{\sharp _{A}}T_{k}\right) ^{\sharp _{A}}\right) \right\Vert _{A} \\
&=&\left\Vert \sum\limits_{k=1}^{n}\left( \frac{\alpha }{4}TT^{\sharp
_{A}}+\left( \frac{\alpha }{4}+\beta \right) T^{\sharp _{A}}T_{k}\right)
\right\Vert _{A} \quad 
\text{(using }\left\Vert T\right\Vert _{A} =\left\Vert T^{\sharp
_{A}}\right\Vert _{A}.
\end{eqnarray*}
\end{proof}

Setting $n=1$ in Theorem \ref{p--1}, we get
\begin{corollary}
\label{CL} \bigskip Let $T\in {\mathcal{B}}_{A}\left( {\mathcal{H}}\right) $%
. Then%
\begin{equation*}
\left\Vert T\right\Vert _{A_{\alpha ,\beta }}^{2}\geq \left\Vert \left( 
\frac{\alpha }{8}TT^{\sharp _{A}}+\left( \frac{\alpha }{8}+\frac{\beta }{2}%
\right) T^{\sharp _{A}}T\right) \right\Vert _{A}\text{.}
\end{equation*}
\end{corollary}

\begin{remark}
For $\beta =1-\alpha $ and $0\leq \alpha \leq 1$ in Corollary 
\ref{CL}, we get \cite[Theorem 2.8]{BH1}, namely,
\begin{equation*}
\left\Vert T\right\Vert _{A_{\alpha }}^{2}\geq \left\Vert \left( \frac{%
\alpha }{8}\left( TT^{\sharp _{A}}+T^{\sharp _{A}}T\right) +\frac{1-\alpha }{%
2}T^{\sharp _{A}}T\right) \right\Vert _{A}\text{.}
\end{equation*}
\end{remark}

The next result reads as follows.

\begin{theorem}
\label{TH5}\bigskip Let $\mathbf{T}=\left( T_{1},\cdots ,T_{n}\right) \in {
\mathcal{B}}_{A}\left( {\mathcal{H}}\right) ^{n}$ be an $n$-tuple. Then
\begin{equation*}
\left\Vert \mathbf{T}\right\Vert _{A_{\alpha ,\beta }}^{2}\leq \sqrt{n}
\omega _{A}\left( \alpha \mathbf{T} \mathbf{T}^{\sharp _{A}}+\beta \mathbf{T}
^{\sharp _{A}} \mathbf{T}\right) \text{,} \quad \textit{where } \mathbf{T}^{\sharp _{A}}= \left( T_{1}^{\sharp_A},\cdots ,T_{n}^{\sharp_A}\right).
\end{equation*}
\end{theorem}

\begin{proof}
Let $x\in {\mathcal{H}}$ with $\left\Vert x\right\Vert _{A}=1$. By using the 
$A$-Cauchy-Schwarz inequality, we get%
\begin{eqnarray*}
\left\Vert \mathbf{T}\right\Vert _{A_{\alpha ,\beta }}^{2}
&=&\sup_{\left\Vert x\right\Vert _{A}=1}\sum\limits_{k=1}^{n}\left( \alpha
\left\vert \left\langle T_{k}x,x\right\rangle _{A}\right\vert ^{2}+\beta
\left\Vert T_{k}x\right\Vert _{A}^{2}\right) \\
&=&\sup_{\left\Vert x\right\Vert _{A}=1}\left( \sum\limits_{k=1}^{n}\left(
\alpha \left\vert \left\langle x,T_{k}^{\sharp _{A}}x\right\rangle
_{A}\right\vert ^{2}+\beta \left\Vert T_{k}x\right\Vert _{A}^{2}\right)
\right) \\
&\leq &\sup_{\left\Vert x\right\Vert _{A}=1}\left(
\sum\limits_{k=1}^{n}\left( \alpha \left\Vert T_{k}^{\sharp
_{A}}x\right\Vert _{A}^{2}+\beta \left\Vert T_{k}x\right\Vert
_{A}^{2}\right) \right) \\
&=&\sup_{\left\Vert x\right\Vert _{A}=1}\left( \sum\limits_{k=1}^{n}\left(
\alpha \left\langle T_{k}T_{k}^{\sharp _{A}}x,x\right\rangle _{A}+\beta
\left\langle T_{k}^{\sharp _{A}}T_{k}x,x\right\rangle _{A}\right) \right) \\
&=&\sup_{\left\Vert x\right\Vert _{A}=1}\left(
\sum\limits_{k=1}^{n}\left\langle \left( \alpha T_{k}T_{k}^{\sharp
_{A}}+\beta T_{k}^{\sharp _{A}}T_{k}\right) x,x\right\rangle _{A}\right) \\
&\leq &\sqrt{n}\sup_{\left\Vert x\right\Vert _{A}=1}\left(
\sum\limits_{k=1}^{n}\left\vert \left\langle \left( \alpha
T_{k}T_{k}^{\sharp _{A}}+\beta T_{k}^{\sharp _{A}}T_{k}\right)
x,x\right\rangle _{A}\right\vert ^{2}\right) ^{{1}/{2}} 
\text{(by } \text{Cauchy--Schwarz)} \\
&=&\sqrt{n}\omega _{A}\left( \alpha \mathbf{T}^{\sharp _{A}}\mathbf{T}+\beta 
\mathbf{TT}^{\sharp _{A}}\right) \text{.}
\end{eqnarray*}
\end{proof}

Using Theorem \ref{TH5}, we deduce the following two corollaries.

\begin{corollary}
Let $\mathbf{T}=\left( T_{1},\cdots ,T_{n}\right) \in {\mathcal{B}}%
_{A}\left( {\mathcal{H}}\right) ^{n}$ be an $n$-tuple. Then
\begin{equation*}
\omega _{A}^{2}\left( \mathbf{T}\right) \leq \inf\limits_{\alpha ,\beta }%
\frac{\sqrt{n}}{\alpha +\beta }\omega _{A}\left( \alpha \mathbf{T}^{\sharp
_{A}}\mathbf{T}+\beta \mathbf{TT}^{\sharp _{A}}\right) \text{.}
\end{equation*}
\end{corollary}

\begin{proof}
Using Theorem \ref{TH1} (i), i.e., $\omega _{A}\left( \mathbf{T}\right)
\leq \frac{1}{\sqrt{\alpha +\beta }}\left\Vert \mathbf{T}\right\Vert
_{A_{\alpha ,\beta }}$, we obtain
\begin{equation*}
\omega _{A}^{2}\left( \mathbf{T}\right) \leq \frac{\sqrt{n}}{\alpha +\beta } \omega _{A}\left( \alpha \mathbf{T}^{\sharp _{A}}\mathbf{T}+\beta \mathbf{TT} ^{\sharp _{A}}\right) \text{}
\end{equation*}
and then taking the infimum over $\alpha ,\beta $, we get the desired result.
\end{proof}

Setting $n=1$, we get the $A$-numerical radius bound, which is a refinement of (\ref{2}).

\begin{corollary}
Let $T\in {\mathcal{B}}_{A}\left( {\mathcal{H}}\right) $. Then
\begin{eqnarray*}
\omega _{A}^{2}\left( T\right) &\leq &\inf\limits_{\alpha ,\beta }\frac{1}{
\alpha +\beta }\left\Vert \alpha T^{\sharp _{A}}T+\beta TT^{\sharp
_{A}}\right\Vert _{A} 
\leq \frac{1}{2}\left\Vert T^{\sharp _{A}}T+TT^{\sharp _{A}}\right\Vert
_{A}\text{.}
\end{eqnarray*}
\end{corollary}

 We now obtain an inequality for the new seminorm, and from which we deduce related $A$-Euclidean operator radius inequalities.

\begin{theorem}
\label{TH6}\bigskip\ Let $\mathbf{T}=\left( T_{1},\cdots ,T_{n}\right) \in {%
\mathcal{B}}_{A}\left( {\mathcal{H}}\right) ^{n}$ be an $n$-tuple. Then
\begin{equation*}
\left\Vert \mathbf{T}\right\Vert _{A_{\alpha ,\beta }}^{2}\leq \left\Vert
\sum\limits_{k=1}^{n}\left( \left( \frac{\alpha }{2}+\beta \right)
T_{k}^{\sharp _{A}}T_{k}+\frac{\alpha }{2}T_{k}T_{k}^{\sharp _{A}}\right)
\right\Vert _{A}\text{.}
\end{equation*}
\end{theorem}

\begin{proof}
Let $x\in {\mathcal{H}}$ with $\left\Vert x\right\Vert _{A}=1$. We have%
\begin{eqnarray*}
\left\vert \left\langle T_{k}x,x\right\rangle _{A}\right\vert ^{2}
&=&\left\vert \left\langle T_{k}x,x\right\rangle _{A}\right\vert \left\vert
\left\langle x,T_{k}^{\sharp _{A}}x\right\rangle _{A}\right\vert \\
&\leq &\left\Vert T_{k}x\right\Vert _{A}\left\Vert T_{k}^{\sharp
_{A}}x\right\Vert _{A} \quad
\text{(by the }A\text{-Cauchy-Schwary inequality)} \\
&=&\left\langle T_{k}^{\sharp _{A}}T_{k}x,x\right\rangle _{A}^{\frac{1}{2}%
}\left\langle T_{k}^{\sharp _{A}}T_{k}x,x\right\rangle _{A}^{{1}/{2}} \\
&\leq &\frac{1}{2}\left( \left\langle T_{k}^{\sharp
_{A}}T_{k}x,x\right\rangle _{A}+\left\langle T_{k}T_{k}^{\sharp
_{A}}x,x\right\rangle _{A}\right) \quad
\text{(by AM-GM inequality)} \\
&=&\frac{1}{2}\left\langle \left( T_{k}^{\sharp
_{A}}T_{k}+T_{k}T_{k}^{\sharp _{A}}\right) x,x\right\rangle \text{.}
\end{eqnarray*}
Also, we have
$\left\Vert T_{k}x\right\Vert _{A}^{2} = \left\langle
T_{k}x,T_{k}x\right\rangle _{A} 
= \left\langle T_{k}^{\sharp _{A}}T_{k}x,x\right\rangle _{A}\text{.}$
Therefore, 
\begin{eqnarray*}
\left\Vert \mathbf{T}\right\Vert _{A_{\alpha ,\beta }}^{2}
&=&\sup_{\left\Vert x\right\Vert _{A}=1}\sum\limits_{k=1}^{n}\left( \alpha
\left\vert \left\langle T_{k}x,x\right\rangle _{A}\right\vert ^{2}+\beta
\left\Vert T_{k}x\right\Vert _{A}^{2}\right) \\
&\leq &\sup_{\left\Vert x\right\Vert _{A}=1}\sum\limits_{k=1}^{n}\left( 
\frac{\alpha }{2}\left\langle \left( T_{k}^{\sharp
_{A}}T_{k}+T_{k}T_{k}^{\sharp _{A}}\right) x,x\right\rangle +\beta
\left\langle T_{k}^{\sharp _{A}}T_{k}x,x\right\rangle \right) \\
&=&\sup_{\left\Vert x\right\Vert _{A}=1}\sum\limits_{k=1}^{n}\left\langle
\left( \left( \frac{\alpha }{2}+\beta \right) T_{k}^{\sharp _{A}}T_{k}+\frac{%
\alpha }{2}T_{k}T_{k}^{\sharp _{A}}\right) x,x\right\rangle _{A} \\
&=&\omega _{A}\left( \sum\limits_{k=1}^{n}\left( \left( \frac{\alpha }{2}%
+\beta \right) T_{k}^{\sharp _{A}}T_{k}+\frac{\alpha }{2}T_{k}T_{k}^{\sharp
_{A}}\right) \right) \\
&=&\left\Vert \sum\limits_{k=1}^{n}\left( \left( \frac{\alpha }{2}+\beta
\right) T_{k}^{\sharp _{A}}T_{k}+\frac{\alpha }{2}T_{k}T_{k}^{\sharp
_{A}}\right) \right\Vert _{A}\text{.}
\end{eqnarray*}
\end{proof}

Using Theorem \ref{TH6}, we deduce a generalization and refinement of \cite[Corollary 3.6]{MOS}.

\begin{corollary}
Let $\mathbf{T}=\left( T_{1},\cdots ,T_{n}\right) \in {\mathcal{B}}
_{A}\left( {\mathcal{H}}\right) ^{n}$ be an $n$-tuple. Then
\begin{eqnarray*}
\omega _{A}\left( \mathbf{T}\right) &\leq &\inf\limits_{\alpha ,\beta }\frac{%
1}{\sqrt{\alpha +\beta }}\left\Vert \sum\limits_{k=1}^{n}\left( \left( \frac{%
\alpha }{2}+\beta \right) T_{k}^{\sharp _{A}}T_{k}+\frac{\alpha }{2}%
T_{k}T_{k}^{\sharp _{A}}\right) \right\Vert _{A}^{{1}/{2}} \\
&\leq &\frac{1}{\sqrt{2}}\left\Vert \sum\limits_{k=1}^{n}\left(
T_{k}^{\sharp _{A}}T_{k}+T_{k}T_{k}^{\sharp _{A}}\right) \right\Vert ^{\frac{%
1}{2}}\text{.}
\end{eqnarray*}
\end{corollary}

\begin{proof} The second inequality follows from the case $\alpha =1,\beta =0$.
Following Theorem \ref{TH1} (i) and Theorem \ref{TH6}, it follows that 
\begin{eqnarray*}
\omega _{A}\left( \mathbf{T}\right) &\leq &\frac{1}{\sqrt{\alpha +\beta} }%
\left\Vert \mathbf{T}\right\Vert _{A_{\alpha ,\beta }} 
\leq \frac{1}{\sqrt{\alpha +\beta} }\left\Vert \sum\limits_{k=1}^{n}\left( \left( 
\frac{\alpha }{2}+\beta \right) T_{k}^{\sharp _{A}}T_{k}+\frac{\alpha }{2}%
T_{k}T_{k}^{\sharp _{A}}\right) \right\Vert _{A}^{{1}/{2}}\text{.}
\end{eqnarray*}
Taking the infimum over all $\alpha ,\beta $, we get 
\begin{equation*}
\omega _{A}\left( \mathbf{T}\right) \leq \inf\limits_{\alpha ,\beta }\frac{1%
}{\sqrt{\alpha +\beta }}\left\Vert \sum\limits_{k=1}^{n}\left( \left( \frac{%
\alpha }{2}+\beta \right) T_{k}^{\sharp _{A}}T_{k}+\frac{\alpha }{2}%
T_{k}T_{k}^{\sharp _{A}}\right) \right\Vert _{A}^{{1}/{2}}\text{.}
\end{equation*}
\end{proof}
\begin{remark}
For $n=1$, we get 
\begin{eqnarray*}
\omega _{A}^{2}\left( T\right) &\leq &\inf\limits_{\alpha ,\beta }\frac{1}{
{\alpha +\beta} }\left\Vert \left( \left( \frac{\alpha }{2}+\beta \right)
T^{\sharp _{A}}T+\frac{\alpha }{2}TT^{\sharp _{A}}\right) \right\Vert _{A} 
\leq \frac{1}{2}\left\Vert T^{\sharp _{A}}T+TT^{\sharp _{A}}\right\Vert
_{A}\text{,}
\end{eqnarray*}
which improves (\ref{2}).
\end{remark}

The next inequality is as follows:

\begin{theorem}
\label{TH7}\bigskip Let $\mathbf{T}=\left( T_{1},\cdots ,T_{n}\right) \in {
\mathcal{B}}_{A}\left( {\mathcal{H}}\right) ^{n}$ be an $n$-tuple. Then
\begin{equation*}
\left\Vert \mathbf{T}\right\Vert _{A_{\alpha ,\beta }}^{2}\leq \sqrt{n}%
\left( \omega _{A}\left( \left( \frac{\alpha }{4}+\beta \right) \mathbf{T}%
^{\sharp _{A}}\mathbf{T}+\frac{\alpha }{4}\mathbf{TT}^{\sharp _{A}}\right) +%
\frac{\alpha }{2}\omega _{A}\left( \mathbf{T}^{2}\right) \right) \text{.}
\end{equation*}
\end{theorem}

\begin{proof}
Let $x\in {\mathcal{H}}$ with $\left\Vert x\right\Vert _{A}=1$. We have 
\begin{eqnarray*}
\left\vert \left\langle T_{k}x,x\right\rangle _{A}\right\vert ^{2}
&=&\left\vert \left\langle T_{k}x,x\right\rangle _{A}\right\vert \left\vert
\left\langle x,T_{k}^{\sharp _{A}}x\right\rangle _{A}\right\vert \\
&\leq &\frac{1}{2}\left( \left\Vert T_{k}x\right\Vert _{A}\left\Vert
T_{k}^{\sharp _{A}}x\right\Vert _{A}+\left\vert \left\langle
T_{k}x,T_{k}^{\sharp _{A}}x\right\rangle _{A}\right\vert \right) \quad
\text{(by Lemma \ref{L1})} \\
&\leq &\frac{1}{4}\left( \left\Vert T_{k}x\right\Vert _{A}^{2}+\left\Vert
T_{k}^{\sharp _{A}}x\right\Vert _{A}^{2}\right) +\frac{1}{2}\left\vert
\left\langle T_{k}^{2}x,x\right\rangle _{A}\right\vert \quad
\text{(by the AM-GM inequality)} \\
&=&\frac{1}{4}\left\langle \left( T_{k}^{\sharp
_{A}}T_{k}+T_{k}T_{k}^{\sharp _{A}}\right) x,x\right\rangle _{A}+\frac{1}{2}%
\left\vert \left\langle T_{k}^{2}x,x\right\rangle _{A}\right\vert \text{.}
\end{eqnarray*}

Therefore, 
\begin{eqnarray*}
&&\sum\limits_{k=1}^{n}\left( \alpha \left\vert \left\langle
T_{k}x,x\right\rangle _{A}\right\vert ^{2}+\beta \left\Vert
T_{k}x\right\Vert _{A}^{2}\right) \\
&\leq &\sum\limits_{k=1}^{n}\left( \left\langle \frac{\alpha }{4}\left(
T_{k}^{\sharp _{A}}T_{k}+T_{k}T_{k}^{\sharp _{A}}\right) x,x\right\rangle
_{A}+\frac{\alpha }{2}\left\vert \left\langle T_{k}^{2}x,x\right\rangle
_{A}\right\vert +\beta \left\langle T_{k}^{\sharp _{A}}T_{k}x,x\right\rangle
_{A}\right) \\
&=&\sum\limits_{k=1}^{n}\left\langle \left( \frac{\alpha }{4}\left(
T_{k}^{\sharp _{A}}T_{k}+T_{k}T_{k}^{\sharp _{A}}\right) +\beta
T_{k}^{\sharp _{A}}T_{k}\right) x,x\right\rangle _{A}+\frac{\alpha }{2}%
\sum\limits_{k=1}^{n}\left\vert \left\langle T_{k}^{2}x,x\right\rangle
_{A}\right\vert \\
&=&\sum\limits_{k=1}^{n}\left\langle \left( \left( \frac{\alpha }{4}+\beta
\right) T_{k}^{\sharp _{A}}T_{k}+\frac{\alpha }{4}T_{k}T_{k}^{\sharp
_{A}}\right) x,x\right\rangle +\frac{\alpha }{2}\sum\limits_{k=1}^{n}\left%
\vert \left\langle T_{k}^{2}x,x\right\rangle \right\vert \\
&\leq &\sqrt{n}\left( \left( \sum\limits_{k=1}^{n}\left\vert \left\langle
\left( \left( \frac{\alpha }{4}+\beta \right) T_{k}^{\sharp _{A}}T_{k}+\frac{%
\alpha }{4}T_{k}T_{k}^{\sharp _{A}}\right) x,x\right\rangle _{A}\right\vert
^{2}\right) ^{{1}/{2}}+\frac{\alpha }{2}\left(
\sum\limits_{k=1}^{n}\left\vert \left\langle T_{k}^{2}x,x\right\rangle
_{A}\right\vert ^{2}\right) ^{{1}/{2}}\right) \\
&&\text{(by the }A\text{-Cauchy-Schwarz inequality)} \\
&\leq &\sqrt{n}\left( \omega _{A}\left( \left( \frac{\alpha }{4}+\beta
\right) \mathbf{T}^{\sharp _{A}}\mathbf{T}+\frac{\alpha }{4}\mathbf{TT}%
^{\sharp _{A}}\right) +\frac{\alpha }{2}\omega _{A}\left( \mathbf{T}%
^{2}\right) \right)\text{.}
\end{eqnarray*}
Taking the supremum over all $x\in {\mathcal{H}}$ with $\left\Vert
x\right\Vert _{A}=1$, we get%
\begin{equation*}
\left\Vert \mathbf{T}\right\Vert _{A_{\alpha ,\beta }}^{2}\leq \sqrt{n}%
\left( \omega _{A}\left( \left( \frac{\alpha }{4}+\beta \right) \mathbf{T}%
^{\sharp _{A}}\mathbf{T}+\frac{\alpha }{4}\mathbf{TT}^{\sharp _{A}}\right) +%
\frac{\alpha }{2}\omega _{A}\left( \mathbf{T}^{2}\right) \right) \text{.}
\end{equation*}
This completes the proof.
\end{proof}

Applying Theorem \ref{TH7}, we derive the following result.

\begin{corollary}
\label{CCC}\bigskip\ Let $\mathbf{T}=\left( T_{1},\cdots ,T_{n}\right) \in {%
\mathcal{B}}_{A}\left( {\mathcal{H}}\right) ^{n}$ be an $n$-tuple. Then
\begin{eqnarray*}
\omega _{A}^{2}\left( \mathbf{T}\right) &\leq &\inf\limits_{\alpha ,\beta }%
\frac{\sqrt{n}}{\alpha +\beta }\left\{ \omega _{A}\left( \left( \frac{\alpha 
}{4}+\beta \right) \mathbf{T}^{\sharp _{A}}\mathbf{T}+\frac{\alpha }{4}%
\mathbf{TT}^{\sharp _{A}}\right) +\left. \frac{\alpha }{2}\omega _{A}\left( 
\mathbf{T}^{2}\right) \right\} \right. \\
&\leq &\sqrt{n}\left( \frac{1}{4}\omega _{A}\left( \mathbf{T}^{\sharp _{A}}%
\mathbf{T}+\mathbf{TT}^{\sharp _{A}}\right) +\frac{1}{2}\omega _{A}\left( 
\mathbf{T}^{2}\right) \right) \text{.}
\end{eqnarray*}
\end{corollary}

\begin{remark}
From Corollary \ref{CCC}, we have 
if $A\mathbf{T}^{2}=0$ (i.e., $AT_{k}^{2}=0$ for all $k=1,\cdots,n$), then
\begin{equation*}
\omega _{A}^{2}\left( \mathbf{T}\right) \leq \frac{\sqrt{n}}{4}\omega
_{A}\left( \mathbf{T}^{\sharp _{A}}\mathbf{T}+\mathbf{TT}^{\sharp
_{A}}\right) \text{.}
\end{equation*}
\end{remark}

For the case $n=1$ in Corollary \ref{CCC}, we get 

\begin{corollary}
\label{CX}\bigskip\ Let $T\in {\mathcal{B}}_{A}\left( {\mathcal{H}}\right) $%
. Then%
\begin{equation*}
\omega _{A}^{2}\left( T\right) \leq \inf\limits_{\alpha ,\beta }\frac{1}{%
\alpha +\beta }\left\{ \left\Vert \left( \frac{\alpha }{4}+\beta \right)
T^{\sharp _{A}}T+\frac{\alpha }{4}TT^{\sharp _{A}}\right\Vert _{A}\right.
+\left. \frac{\alpha }{2}\omega _{A}\left( T^{2}\right) \right\} \text{.}
\end{equation*}
\end{corollary}

\begin{remark}
The inequality in Corollary \ref{CX} refines the
inequality (\ref{4}). Indeed, we have 
\begin{eqnarray*}
\omega _{A}^{2}\left( T\right) &\leq &\inf\limits_{\alpha ,\beta }\frac{1}{%
\alpha +\beta }\left\{ \left\Vert \left( \frac{\alpha }{4}+\beta \right)
T^{\sharp _{A}}T+\frac{\alpha }{4}TT^{\sharp _{A}}\right\Vert _{A}\right.
+\left. \frac{\alpha }{2}\omega _{A}\left( T^{2}\right) \right\} \\
&\leq &\frac{1}{4}\left\Vert T^{\sharp _{A}}T+TT^{\sharp _{A}}\right\Vert
_{A}+\frac{1}{2}\omega _{A}\left( T^{2}\right) \text{.}
\end{eqnarray*}
\end{remark}

Finally, we prove the following inequality for the new seminorm.

\begin{theorem}
\label{TTTT}\bigskip\ Let $\mathbf{T}=\left( T_{1},\cdots ,T_{n}\right) \in {
\mathcal{B}}_{A}\left( {\mathcal{H}}\right) ^{n}$, $\mathbf{S}=\left(
S_{1},\cdots ,S_{n}\right) $ $\in {\mathcal{B}}_{A}\left( {\mathcal{H}}%
\right) ^{n}$. Then
\begin{equation*}
\left\Vert \mathbf{S}^{\sharp _{A}}\mathbf{T}\right\Vert _{A_{\alpha ,\beta
}}^{2}\leq \frac{\alpha }{2}\left\Vert \sum\limits_{k=1}^{n}\left( \left(
T_{k}^{\sharp _{A}}T_{k}\right) ^{2}+\left( S_{k}^{\sharp _{A}}S_{k}\right)
^{2}\right) \right\Vert _{A}+\beta \sum\limits_{k=1}^{n}\left\Vert
S_{k}^{\sharp _{A}}T_{k}\right\Vert _{A}^{2}\text{.}
\end{equation*}
\end{theorem}

\begin{proof}
Let $x\in {\mathcal{H}}$ and $\|x\|_A=1$. Then, we have
\begin{eqnarray*}
\sum\limits_{k=1}^{n}\left\vert \left\langle S_{k}^{\sharp
_{A}}T_{k}x,x\right\rangle _{A}\right\vert ^{2}
&=&\sum\limits_{k=1}^{n}\left\vert \left\langle T_{k}x,S_{k}x\right\rangle
_{A}\right\vert ^{2} \\
&\leq &\sum\limits_{k=1}^{n}\left\Vert T_{k}x\right\Vert _{A}^{2}\left\Vert
S_{k}x\right\Vert _{A}^{2} \quad
\text{(by the }A\text{-Cauchy-Schwarz inequality)} \\
&=&\sum\limits_{k=1}^{n}\left\langle T_{k}^{\sharp
_{A}}T_{k}x,x\right\rangle _{A}\left\langle S_{k}^{\sharp
_{A}}S_{k}x,x\right\rangle _{A} \\
&\leq &\sum\limits_{k=1}^{n}\frac{1}{2}\left( \left\langle T_{k}^{\sharp
_{A}}T_{k}x,x\right\rangle _{A}^{2}+\left\langle S_{k}^{\sharp
_{A}}S_{k}x,x\right\rangle _{A}^{2}\right) 
\text{(by the AM-GM inequality)} \\
&\leq &\sum\limits_{k=1}^{n}\frac{1}{2}\left( \left\langle \left(
T_{k}^{\sharp _{A}}T_{k}\right) ^{2}x,x\right\rangle +\left\langle \left(
S_{k}^{\sharp _{A}}S_{k}\right) ^{2}x,x\right\rangle _{A}\right) \quad
\text{(by Lemma \ref{L2})} \\
&=&\sum\limits_{k=1}^{n}\frac{1}{2}\left\langle \left( \left( T_{k}^{\sharp
_{A}}T_{k}\right) ^{2}+\left( S_{k}^{\sharp _{A}}S_{k}\right) ^{2}\right)
x,x\right\rangle _{A} \\
&\leq &\frac{1}{2}\left\Vert \sum\limits_{k=1}^{n}\left( \left(
T_{k}^{\sharp _{A}}T_{k}\right) ^{2}+\left( S_{k}^{\sharp _{A}}S_{k}\right)
^{2}\right) \right\Vert _{A}\text{.}
\end{eqnarray*}
We also have $\sum_{k=1}^n \left\Vert
S_{k}^{\sharp _{A}}T_{k}x\right\Vert _{A}^{2}\leq \sum_{k=1}^n  \left\Vert
 S_{k}^{\sharp _{A}}T_{k}\right\Vert_{A}^{2}.$
Therefore,
\begin{eqnarray*}
\sum\limits_{k=1}^{n}\left( \alpha \left\vert \left\langle S_{k}^{\sharp
_{A}}T_{k}x,x\right\rangle _{A}\right\vert ^{2}+\beta \left\Vert
S_{k}^{\sharp _{A}}T_{k}x\right\Vert _{A}^{2}\right) 
\leq \frac{\alpha }{2}\left\Vert \sum\limits_{k=1}^{n}\left( \left(
T_{k}^{\sharp _{A}}T_{k}\right) ^{2}+\left( S_{k}^{\sharp _{A}}S_{k}\right)
^{2}\right) \right\Vert +\beta \sum\limits_{k=1}^{n}\left\Vert S_{k}^{\sharp
_{A}}T_{k}\right\Vert ^{2}\text{.}
\end{eqnarray*}
Taking the supremum over all $x\in {\mathcal{H}}$ with $\left\Vert
x\right\Vert _{A}=1$, we get as desired.
\end{proof}

Applying Theorem \ref{TTTT} and Theorem \ref{TH1} (i), we deduce a bound for the $A$-Euclidean operator radius of the product of two $n$-tuples:

\begin{corollary}
\label{CW} Let $\mathbf{T}=\left( T_{1},\cdots ,T_{n}\right) \in {\mathcal{B}%
}_{A}\left( {\mathcal{H}}\right) ^{n}$, $\mathbf{S}=\left( S_{1},\cdots
,S_{n}\right) $ $\in {\mathcal{B}}_{A}\left( {\mathcal{H}}\right) ^{n}$. Then%
\begin{eqnarray*}
\omega _{A}^{2}\left( \mathbf{S}^{\sharp _{A}}\mathbf{T}\right) &\leq
&\inf_{\alpha ,\beta }\frac{1}{\alpha +\beta }\left\{ \frac{\alpha }{2}%
\left\Vert \sum\limits_{k=1}^{n}\left( \left( T_{k}^{\sharp
_{A}}T_{k}\right) ^{2}+\left( S_{k}^{\sharp _{A}}S_{k}\right) ^{2}\right)
\right\Vert _{A}+\beta \sum\limits_{k=1}^{n}\left\Vert S_{k}^{\sharp
_{A}}T_{k}\right\Vert _{A}^{2}\right\} \\
&\leq &\frac{1}{2}\left\Vert \sum\limits_{k=1}^{n}\left( T_{k}^{\sharp
_{A}}T_{k}\right) ^{2}+\left( S_{k}^{\sharp _{A}}S_{k}\right)
^{2}\right\Vert _{A}\text{.}
\end{eqnarray*}
\end{corollary}

Setting $n=1$ in Theorem \ref{TTTT}, we get

\begin{corollary}
Let $T,S\in {\mathcal{B}}_{A}\left( {\mathcal{H}}\right) $. Then%
\begin{equation*}
\left\Vert S^{\sharp _{A}}T\right\Vert _{A_{\alpha ,\beta }}^{2}\leq \frac{%
\alpha }{2}\left\Vert \left( T^{\sharp _{A}}T\right) ^{2}+\left( S^{\sharp
_{A}}S\right) ^{2}\right\Vert_A +\beta \left\Vert S^{\sharp _{A}}T\right\Vert
_{A}^{2}\text{.}
\end{equation*}
\end{corollary}

Again, from Corollary \ref{CW}, we deduce a new refinement of (\ref{6}) (for the case $n=2$).

\begin{corollary}
Let $T,S\in {\mathcal{B}}_{A}\left( {\mathcal{H}}\right) $. Then%
\begin{eqnarray*}
\omega ^{2}\left( S^{\sharp _{A}}T\right) &\leq& \inf_{\alpha ,\beta }\frac{1
}{\alpha +\beta }\left\{ \frac{\alpha }{2}\left\Vert \left( T^{\sharp
_{A}}T\right) ^{2}+\left( S^{\sharp _{A}}S\right) ^{2}\right\Vert +\beta
\left\Vert S^{\sharp _{A}}T\right\Vert _{A}^{2}\right\}\\ 
&\leq& \frac{1}{2}\left\Vert \left( T^{\sharp _{A}}T\right) ^{2}+\left(
S^{\sharp _{A}}S\right) ^{2}\right\Vert _{A}\text{.}
\end{eqnarray*}
\end{corollary}

\bigskip
\section*{Declarations}

\noindent  \textbf{Conflict of interest.} The authors declare that there is no conflict of interest.

\noindent  \textbf{Data availability.} No data was used for the research described in the article.

\end{document}